\documentclass[12pt]{article}
\usepackage{latexsym}
\usepackage[english]{babel}
\usepackage[letterpaper,top=2cm,bottom=2cm,left=3cm,right=3cm,marginparwidth=1.75cm]{geometry}
\usepackage{amsfonts}
\usepackage{amsmath}
\usepackage{amssymb}
\usepackage{amsthm}
\usepackage{graphicx}
\usepackage[colorlinks=true, allcolors=blue]{hyperref}
\usepackage[switch, modulo]{lineno}

\newtheorem{theorem}{Theorem}
\newtheorem{lemma}[theorem]{Lemma}

\newtheorem{remark}[theorem]{Remark}
\newtheorem{corollary}[theorem]{Corollary}
\usepackage[T1]{fontenc}
\def\R{\mathbb{R}}
\def\N{\mathbb{N}}
\def\e{\varepsilon}

\title{From indirect  to direct taxis by fast reaction limit}
\author{J. Ignacio Tello $^*$ and Dariusz Wrzosek $^{**}$  \\
$^*$ {\small Department of Fundamental Mathematics. School of Sciences}\\
{ \small National University of Distance Education. Madrid, Spain} \\
$^{**}$ {\small Institute of Applied Mathematics and Mechanics},\\
 {\small University of Warsaw}}

\begin{document}
\maketitle
 
\begin{abstract}
 Many ecological population models consider taxis as the directed movement of animals in response to a stimulus. The taxis is named  direct  if the animals are   guided by the density gradient of some other population or indirect  if they are guided by the density of a chemical secreted by individuals of the other population.  Let $u$ and $v$ denote the densities of two populations and $w$ the density of the chemical secreted by individuals in the $v$ population. We consider a bounded,  open  set $\Omega \subset \mathbb{R}^N$ with regular boundary and prove that for  the space dimension $N\leq 2$ the solution to the  Lotka-Volterra competition  model with repulsive   indirect  taxis and homogeneous Neumann boundary conditions
 $$ 
 \left\{ \begin{array}{l} 
   u_t - d_u\Delta u  = \chi \nabla \cdot u \nabla w +\mu_1u(1-u-a_1v)\,,\\ [2mm] 
   v_t - d_v\Delta v = \mu_2v(1-v-a_2u)\,,\\ [2mm]
  \varepsilon (  w_t - d_w\Delta w )=    v-    w\, , 
 \end{array}
\right.$$
 converges to the solution of  repulsive  direct-taxis model: 
 $$  \left\{ \begin{array}{l} 
   u_t - d_u\Delta u = \chi \nabla \cdot u \nabla v +\mu_1u(1-u-a_1v)\,,\\ [2mm] 
   v_t - d_v\Delta v = \mu_2v(1-v-a_2u)\,
  \end{array}
\right.$$
when $\varepsilon\longrightarrow 0$. For space dimension $N\geq 3$ we use the compactness argument to show that the result holds in some weak sense. A similar result is also proved for a typical prey-predator model with prey taxis and logistic growth of predators.
\end{abstract}

\section{Introduction}
 The following system is an extension of the classical Lotka-Volterra competition model considering the diffusive migration of both species,  with densities $U$ and $V$. The model also includes    avoidance strategy of one of the populations  which moves towards the lower concentration   of the chemical $W$ secreted by the other competitor:
 \begin{align}  \label{U1}
   U_{\tilde{t}} - D_u\Delta U &= \chi_0 \nabla \cdot (U \nabla W) +\alpha_1 U -\alpha_2U^2 -\alpha_3UV \,, \\ \label{U2}
   V_{\tilde{t}} - D_v\Delta V &= \beta_1 V -\beta_2 V^2 -\beta_3UV
   \,,  \\ \label{U3}
   W_{\tilde{t}} - D_w\Delta W &= -\alpha W +\lambda V \,,
\end{align}
 with the homogeneus Neumann boundary conditions on the boundary of a bounded domain $\Omega \subset
\mathbb{R}^N $ with smooth boundary. All parameters, diffusion constants $D_u\,,D_v\,,D_w$ and kinetic parameters $\alpha_i\,,\beta_i$ (for $i=1,2$) as well as the chemotactic sensitivity coefficient $\chi_0$ are assumed to be positive constants.  The system is supplemented with initial conditions for $\tilde{t}=0$ and $\tilde{x}\in \Omega$.
This type of transport mechanism based on a chemical signaling  is called indirect taxis \cite{TWr1}, in contrast to direct taxis, in which the escape direction  is guided by the density gradient of the competitor itself without any intermediary mechanisms 
\begin{align} \label{UU1}
   U_{\tilde{t}} - D_u\Delta U &= \chi \nabla \cdot 
   U \nabla V +\alpha_1 U -\alpha_2U^2 -\alpha_3UV\,,\\ \label{UU2}
   V_{\tilde{t}} - D_v\Delta V &= \beta_1 V -\beta_2 V^2 -\beta_3UV\,. 
\end{align}
The concept of indirect taxis, already introduced in the ecological context in \cite{TWr1}, has been used in a considerable number of works; we refer the reader to only a few of them \cite{Ahnand,Ahnand2,Li, MiWr2, MiWr3} and the references therein. We also refer to surveys \cite{Bellomo1,Bellomo2} on the mathematical theory of chemotaxis and the wide spectrum of applications in the biological and social sciences.  

It turns out that the existence of global in time classical solutions to system (\ref{U1})-(\ref{U3}) is relatively easy to prove (c.f. \cite{TWr}) for any space dimension $N\geq 1$ while for system (\ref{UU1})-(\ref{UU2}) the existence of classical solutions is more complicated to prove and known only for $N\leq 2$. 
In this article we would like to investigate whether solutions of system (\ref{UU1})-(\ref{UU2}) can be obtained in a fast-reaction limit  from the solutions to  system (\ref{U1})-(\ref{U3}). To this end we introduce the dimensionless variables $(u(x,t)\,,v(x,t)\,,w(x,t))$ with $x=\frac{\tilde{x}}{L}\,, t=\frac{\tilde{t}}{\tau}$
$$ u\left(\frac{\tilde{x}}{L},\frac{\tilde{t}}{\tau}\right) =\frac{\alpha_2}{\alpha_1} U(\tilde{x}\,,\tilde{t})\,,\quad v\left(\frac{\tilde{x}}{L},\frac{\tilde{t}}{\tau}\right)=\frac{\beta_2}{\beta_1}V\left(\frac{\tilde{x}}{L},\frac{\tilde{t}}{\tau}\right)\,,\quad w\left(\frac{\tilde{x}}{L},\frac{\tilde{t}}{\tau}\right)=\frac{W(\tilde{x}\,,\tilde{t})}{W^*}$$
where $W^*$ is some reference  density. Setting
$$ d_u=\frac{D_u\tau}{L^2}\,,d_v=\frac{D_v\tau}{L^2}\,, d_w=\frac{D_w\tau}{L^2}\,, \chi=\frac{\tau\chi_0 W^*}{L^2}$$
and 
$$ \mu_1=\alpha_1\tau\,, \mu_2=\beta_1\tau\,,a_1=\frac{\alpha_3\beta_1}{\beta_2\alpha_1}\,,a_2=\frac{\beta_3\alpha_1}{\alpha_2\beta_1}\,,\tilde{\alpha}=\alpha \tau\,, \tilde{\lambda}=\lambda\tau $$
we arrive  at 
\begin{align*} 
   u_t - d_u\Delta u &= \chi \nabla \cdot u \nabla w +\mu_1u(1-u-a_1v)\,,\\ 
   v_t - d_v\Delta v &= \mu_2v(1-v-a_2u)\,,\\
   w_t - d_w\Delta w &= \tilde{\lambda} v- \tilde{\alpha} w\,.
 \end{align*}
Notice that  
$d_u\,, d_v\,, \chi\,, \e\,,\mu_1\,, \mu_2\,,  a_1\,, a_2\,,\tilde{\lambda}\,, \tilde{\alpha} $ are positive constants.
Now, we are in a position to set $d_w=1$ as well as 
$$\tilde{\lambda}=\tilde{\alpha}=\frac{\alpha L^2}{d_w}=\frac{\lambda L^2}{d_w}=\frac{1}{\varepsilon}\,,\quad\mbox{for}\quad \e\in (0,1)\,,$$ 
to refer to the situation in which the rate of kinetic reactions in the $w$-equations significantly exceeds that of diffusion for $\left(\frac{1}{\e}\right)$ large enough. In the limiting case $\e\rightarrow 0 $ we obtain the so-called fast reaction limit studied in a slightly different context than ours (see \cite{Evans} as a benchmark and \cite{Murakawa} containing more resent results).  In this article we want to answer the question whether in any sense it is possible to pass to the limit, letting $\e\rightarrow 0 $,  from the solutions of the system 
 \begin{align} \label{ue1}
   u_{\e,t} - d_u\Delta u_\e &= \chi \nabla \cdot (u_\e \nabla w_\e ) +\mu_1 u_\e(1-u_\e-a_1v_\e)\,,\\ \label{ue2}
   v_{\e,t} - d_v\Delta v_\e &= \mu_2 v_\e(1 -v_\e-a_2u_\e)\,,\\ \label{ue3}
  \e w_{\e,t} - \e \Delta w_\e &=  v_\e- w_\e\,,
 \end{align}
to the solution of 
     \begin{align} \label{uu1}
   u_{t} - d_u\Delta u &= \chi \nabla \cdot (u \nabla v) +\mu_1 u(1-u-a_1v)\,,\\ \label{uu2}
   v_{t} - d_v\Delta v &= \mu_2 v(1 -v-a_2u)\,.
 \end{align}   
with suitably regular initial conditions 
and  homogeneous Neumann boundary conditions  for both systems. It is worth noticing that the limit problem can be viewed as a particular case of the famous Shigesada, Kawasaki and Teramoto system, which has been studied in the past decades in many papers among which we point to \cite{Yagi} and \cite[Sec. 2]{Le} or to \cite{MiWr} where system (\ref{uu1})-(\ref{uu2}) has been studied.

It is also worth mentioning that our problem falls into the broad class of small-parameter methods with the famous Tikhonov theorem as a benchmark (see the monograph \cite{BanLach} for survey and references). This is due to the fact that the passage to the limit with $\varepsilon$ is automatically related to the reduction of the number of equations and is commonly called the quasi-stationary approximation. It is worth emphasizing here that, unlike systems of ordinary differential equations, there is, to the best of our knowledge, no general theory regarding small parameter approximation for reaction-diffusion systems, and even much less is known in this context about reaction-diffusion systems with advection (chemotaxis). 

For $N\leq 2$ the classical unique global-in-time  solution to system (\ref{uu1})-(\ref{uu2})  exists for sufficiently regular initial condition (see \cite{Le}, \cite{Yagi} or \cite{MiWr}). 

It turns out that for the case $N\leq 2$ we can prove the convergence of the full sequence of solutions to system  (\ref{ue1})-(\ref{ue3}) to the solution of the limit problem (cf. Theorem \ref{L22}, Theorem \ref{lnu2} with Corollary \ref{lnu2cor}). Moreover if the following compatibility condition is satisfied by the initial data
\begin{equation}
\label{ic}
 u(\cdot,0)=u_\e(\cdot,0)=u_0(\cdot)\,, v(\cdot,0)=v_\e(\cdot,0)=w_\e(\cdot,0)=v_0(\cdot,0)\,,  
 \end{equation}
 we find  $\e$-dependent linear convergence estimates for any finite time interval $[0\,,T]$
\begin{align} \label{pp1}
 \sup_{t\in[0,T]} \|u(\cdot, t)-u_\e(\cdot, t)\|_2\leq \e \tilde{C}(T)\,, \\ 
 \label{pp4}
\| u- u_\e\|_{L^2(0\,,T:W^{1,2}(\Omega)} \leq \e \tilde{C}(T)\,,\\ \label{pp2}
 \sup_{t\in[0,T]} \|v(\cdot, t)-v_\e(\cdot, t)\|_{W^{1,2}(\Omega)}\leq \e \tilde{C}(T)\,,  \\ \label{pp3}
 \sup_{t\in[0,T]} \|\nabla v(\cdot, t)-\nabla w_\e(\cdot, t)\|_2\leq \e \tilde{C}(T)\,.  
 \end{align}
These estimates allow us to prove that the solution to (\ref{ue1})-(\ref{ue3}) converges in a standard $L^2(\Omega)$ weak sense to the solution to (\ref{uu1})-(\ref{uu2}). 

For $N>2$ we use the compactness method to show in Theorem \ref{Lw22} that,  when choosing a subsequence, the solution of (\ref{ue1})-(\ref{ue3}) converges to a solution of (\ref{uu1})-(\ref{uu2}) in a suitable weak sense. 
The methods used to analyze the relationship between competition models with direct and indirect taxis have been successfully applied in a similar context to the predator-prey model with logistic growth of prey and predator cf. Theorem \ref{TheoPred}.

\begin{remark} 
In general,   ecological models with direct and indirect taxis may have significantly distinct properties, in particular when the stability of a constant steady state is concerned. We refer the reader to \cite{WuWangShi}  where the predator-prey model with direct predator taxis (i.e. repulsive taxis of prey in reaction to the gradient of predator density) is studied with the Rosenzweig-MacArthur kinetics. It turns out that direct predator taxis do not affect the stability of the constant steady state, while in \cite[Remark 3.1]{MiWr3} it is shown that in a corresponding model with indirect taxis the coexistence steady state may become unstable due to indirect taxis and give rise to the Hopf bifurcation. A similar situation occurs for predator-prey models with direct and indirect prey taxis when logistic limitation is not taken into account for predator. For most of such models with realistic predator-prey kinetics, contrary to the case of direct taxis, which in fact stabilizes the steady state; for the case of indirect taxis, the chemotactic sensitivity coefficient $\chi$ large enough destabilizes the coexistence steady state. We refer the reader to \cite{Ahnand} and the references therein for further details.
\end{remark}

In the notation we will sometimes drop the arguments of time- and space-dependent functions writing $v(\cdot, t)$ or $v(t)$ or just $v$ instead of $v(x,t)$ etc. depending on the length of formulae in which they appear. For the same reason, sometimes we write $\int_0^T\int_\Omega v $ instead of $\int_0^T\int_\Omega v(x,t)dxdt$. The norm in space $L^p(\Omega)$, $p\in [1\,,\infty]$ will be denoted by $\Vert \cdot\Vert_p$, the norm in Sobolev space $W^{1,p}(\Omega)$ by $\Vert \cdot\Vert_{1,p}$ and the set $\Omega\times(0,T)$ by $\Omega_T$.

The paper is organized as follows. In Section 2 we derive $\varepsilon$- independent estimates that are valid for any dimension of space. Section 3 contains estimates for dimension $N\leq 2$. In Section 4 we first deal with the convergence of solutions when $\varepsilon \rightarrow 0$  for the case $N\leq 2$, then,  with the convergence of a subsequence of to a weak solution by means of compactness arguments for $N >2$. Section 5 is devoted to extending the results obtained for the competition model to the case of the predator-prey model with indirect and direct taxis.

\section{A priori  estimates (for any dimension)}
In this section we derive various estimates of the solutions to the system (\ref{ue1}) - (\ref{ue3}) which are uniform with respect to $\e$. Along the article,   we will use the following well-known differential inequalities quoted in the following lemmas for the convenience of the reader. 
\begin{lemma}\label{LA} \cite[Lemma 3.4]{Stinner}
     Let $T>0 $ or $T=\infty$,  $\tau\in (0\,,T) $ and  
   $h\in L^1_{loc}([0\,,T))$ is a non-negative function satisfying 
\[ \int_{t-\tau}^{t} h(s)ds \leq b\quad\mbox{for all}\;\; t\in [\tau\,,T)\,,
\]
with $b>0$  
     and $y:[0\,,T)\rightarrow \R_+$
a solution to the differential inequality  
\begin{equation} \label{ST1}
\frac{dy}{dt} +a y(t)\leq h(t)\quad\mbox{for a.e.}\quad t\in (0\,,T)
\end{equation}
for $a >0$. Then
\begin{equation} \label{y1}
 y(t)\leq \max\left\{y(0)+b\,,\frac{b}{a\tau} +2b\right\}\quad\mbox{for}\;\; t\in (0,T)\,.
\end{equation} 
\end{lemma}
\begin{lemma}\label{LB}\cite[Lemma 2.3]{Jin}
     Let $T>0 $ or $T=\infty$, $\tau\in (0\,,T)$ and  
     $g,h\in L^1_{loc}([0\,,T)$  are  non-negative functions satisfying 
\[ \int_{t-\tau}^{t} g(s)ds,  \leq \alpha \quad \quad\int_{t-\tau}^{t} h(s)ds \leq \beta, \quad \mbox{for all}\;\; t\in [\tau\,,T)\,,
\]
with $\alpha\,,\beta \geq 0$
and $\xi:[0\,,T)\rightarrow \R_+$
a solution to the differential inequality
\begin{equation} \label{ST11}
\frac{d\xi(t)}{dt} -g(t)\xi(t) +a_1 \xi(t)^{1+\varrho}\leq h(t)\quad\mbox{for}\quad t\in (0\,,T)
\end{equation}
for  $\varrho >0\,,\;a_1>0$. Then
\begin{equation}
\sup_{t\in (0,T)} \xi (t) \leq \varrho \left( \frac{2A}{1+\varrho}\right)^{\frac{1+\varrho}{\varrho}}  +2B 
\end{equation} 
where 
\[A=(\tau a_1)^{-\frac{1}{1+\varrho}}(1+\alpha)^{\frac{1}{1+\varrho}} e^{2\alpha}\,,\quad B=(\tau a_1)^{-\frac{1}{1+\varrho}}\beta^{\frac{1}{1+\varrho}}e^{2\alpha}+2\beta e^{2\alpha} +\xi(0)e^\alpha \,.
\]
\end{lemma}
We would like to emphasize that the formulations that appear in the literature
of these previous lemmas do not include the case $T=\infty$, nevertheless their proofs can easily be extended to the case.
The following lemma concerns the existence of solutions to (\ref{ue1}) - (\ref{ue3}) and basic bounds.
\begin{lemma} \label{L1}
 We assume that $N\geq 1$ and the initial conditions  
 \begin{equation}\label{icreg}
  u_{\e,0}\,,v_{\e,0}\,,w_{\e,0}\in W^{1,q} (\Omega)\,, q>\max\{N\,,4\}   
 \end{equation}
  are nonnegative functions. Then,  for any $\e\in (0,1)$ there exists a global in time unique  classical solution  to (\ref{ue1}) - (\ref{ue3}) defined in $\Omega\times (0\,,+\infty)$ such that for any $T>0$ 
\[(u_\e,v_\e,w_\e)\in (C([0\,,T]:W^{1,q}(\Omega))\cap 
C^{2,1}(\bar{\Omega}\times (0\,,T)))^3\,.\] 
Moreover, there exist constants $\bar{u}_1\,,\bar{v}_1\,,\bar{w}_1$, $C_1(T)$,   $C_2(T)$, $\bar{v}_\infty$ and $\bar{w}_\infty$ such that,  for any $\e\in (0,1)$
\begin{equation}\label{M1}
\sup_{t>0}\left(\Vert u_\e(\cdot,t)\Vert_{1}\right)\leq \bar{u}_1\,,\quad 
\sup_{t>0}\left(\Vert v_\e(t)\Vert_{1}\right)\leq \bar{v}_1\,,\quad \sup_{t>0}\left(\Vert w_\e(t)\Vert_{1}\right)\leq \bar{w}_1, 
\end{equation}
\begin{equation}\label{C1C2}
   \int_{0}^T\int_\Omega u_\e(x,s)^2dxds\leq  C_1(T)\,,\quad \int_{0}^T\int_\Omega v_\e(x,s)^2dxds\leq  C_2(T)\,
\end{equation} 	
and 
\begin{equation}\label{t1}
    \int_{t}^{t+1}\int_\Omega u_\e(x,s)^2dxds\leq  \mu_1\bar{u}_1\,,\; \int_{t}^{t+1}\int_\Omega v_\e(x,s)^2dxds\leq \mu_2\bar{v}_1\,\quad \mbox{for any}\quad t>0,
\end{equation}
\begin{equation}\label{vwinf}
\sup_{t>0}\left(\Vert v_\e(t)\Vert_{\infty}\right)\leq \bar{v}_\infty\,,\quad \sup_{t>0}\left(\Vert w_\e(t)\Vert_{\infty}\right)\leq \bar{w}_\infty\,.
\end{equation}
\end{lemma}
\begin{proof}
The existence of a nonnegative local-in-time unique solution follows from Amann$^{\prime}$s theory \cite[Theorems~14.4 \&~14.6]{Am93} and global existence along with (\ref{M1}) was proved in \cite{TWr}. The inequalities in (\ref{t1}) result directly from the integration of the $u$-equation and $v$-equation with respect to $x$ and $t$, while the first equation in (\ref{vwinf}) is from the comparison with the logistic ordinary differential equation. The bound on $w$ easily follows from the inequality
\[
w_{\e,t} +\Delta w_\e +\frac{1}{\e}w_\e\leq \frac{1}{\e}\bar{v}_\infty
\]
and the comparison with the ordinary differential equation
\[ w_ {\e, t} +\Delta w_\e +\frac{1}{\e}\left(w_\e- \bar{v}_\infty\right)\leq \frac{d\xi}{dt}+\frac{1}{\e}\xi\,.
\] Therefore,
\[w_\e(x,t)\leq \bar{v}_\infty +(\|w_{\e,0}\|_\infty -\bar{v}_\infty)e^{-{\frac{1}{\e}t}}.
\]
\end{proof}
The next lemma provides us with typical energy estimates for the $v$-equation. 
\begin{lemma} \label{base}
 Let $N\geq 1$ and $(u_\e,v_\e, w_\e)$ be a solution to the system (\ref{ue1})-(\ref{ue3}),  then,  there exist  constants  $\bar{v}_2, \ldots, \bar{v}_5$ such that the following bounds are uniform with respect to $\varepsilon\in (0,1) $
\begin{equation}\label{v2}
\sup_{t>0}\Vert v_\e(\cdot,t)\Vert_2\leq \bar{v}_2\,,
\end{equation}
\begin{equation}\label{v3}
\sup_{t>0} \int_\Omega |\nabla v_\e(x,t)|^2dx\leq \bar{v}_3 \,,
\end{equation}
\begin{equation}\label{v4}
\int_{0}^T\int_\Omega |\Delta v_\e(x,s)|^2dxds\leq \bar{v}_4(T)\,
\end{equation}
and 
\begin{equation}\label{v5}
\int_{0}^T\int_\Omega | v_{\e,t}(x,s)|^2dxds\leq \bar{v}_5(T)\,.
\end{equation}
Moreover, there exist  constants $\bar{v}_4^\prime\,,\bar{v}_5^\prime$ independent of $t$ such that 
\begin{align}
\label{v44}
\int_{t}^{t+1}\int_\Omega |\Delta v_\e(x,s)|^2dxds\leq \bar{v}_4^\prime\quad\mbox{for any}\quad t>0\,,\\
\label{vt44}
\int_{t}^{t+1}\int_\Omega |v_{\e,t}(x,s)|^2dxds\leq \bar{v}_5^\prime\quad\mbox{for any}\quad t>0\,.
\end{align}
\end{lemma}
\begin{proof}
   We  
 multiply the $v$-equation  by $v_\e(\cdot,t)$  to obtain, after integration: 
\begin{align}\label{v1}
\frac{1}{2}\frac{d}{dt}\int_\Omega v_\e^2 + d_v\int_\Omega |\nabla v_\e|^2 &=\mu_2\int_\Omega v_\e^2-\mu_2\int_\Omega v_\e^3- a_2\mu_2\int_\Omega u_\e v_\e^2\\ \nonumber
&\leq -\frac{\mu_2}{2}\int_\Omega v_\e^3 +\frac{8\mu_2^3|\Omega|}{27\mu_2^2} \\ \nonumber
&\leq -\frac{\mu_2}{2|\Omega|^{1/2}}\left(\int_\Omega v_\e^2\right)^{\frac{3}{2}} +\frac{8\mu_2^3|\Omega|}{27\mu_2^2} \,.
\end{align}
We may now directly use Lemma \ref{LB} with $\xi(t)= \int_\Omega v_\e(\cdot, t) ^2$, $\varrho=\frac{1}{2}$ and $g=0$ to obtain (\ref{v2}). Next, by multiplying the $v$-equation  by $-\Delta v(\cdot,t)$ for $t>0$, we obtain 
\begin{align*}
    \frac{1}{2}\frac{d}{dt}\int_\Omega |\nabla v_\e|^2 +& d_v\int_\Omega |\Delta v_\e|^2 -\int_\Omega v_\e \Delta v_\e \\ 
    &= -(\mu_2 +1) \int_\Omega v_e \Delta v_\e + \mu_2\int_\Omega v_\e^2 \Delta v_\e +\mu_2\int_\Omega u_\e v_\e\Delta v_\e\,.
\end{align*}
Hence,  using integration by parts and the boundary conditions we obtain 
\begin{align*}
  \frac{1}{2}\frac{d}{dt}\int_\Omega |\nabla v_\e|^2 + &d_v\int_\Omega |\Delta v_\e|^2 + \int_\Omega |\nabla v_\e|^2\\ 
  &\leq (\mu_2+1)\int_\Omega |\Delta v_\e| v -2\mu_2\int_\Omega v|\nabla v_\e|^2 + a_2\mu_2\int_\Omega u_\e v_\e|\Delta v_\e|  
\end{align*}
and then  the Young inequality yields
\begin{equation}\label{gradv}
\frac{1}{2}\frac{d}{dt}\int_\Omega |\nabla v_\e|^2 +\frac{D_v}{2}\int_\Omega |\Delta v_\e|^2 + \int_\Omega |\nabla v_\e|^2\leq \frac{(\mu_2+1)^2}{D_v}\int_\Omega v_\e^2 +\frac{(a_2\mu_2)^2}{D_v}\int_\Omega u_\e^2v_\e^2 \,.
\end{equation}
Whence,  we get the inequality 
\begin{equation}\label{gradvvv}
\frac{1}{2}\frac{d}{dt}\int_\Omega |\nabla v_\e|^2 +\frac{D_v}{2}\int_\Omega |\Delta v_\e|^2 + \int_\Omega |\nabla v_\e|^2\leq \frac{(\mu_2+1)^2}{D_v}\bar{v}_2^2 +\frac{(a_2\mu_2)^2}{D_v}\bar{v}_{\infty}^2\int_\Omega u_\e^2.
\end{equation}
Now,  we are in a position to apply Lemma \ref{LA} and local integrability condition (\ref{t1})  with 
$$\xi (t)=\int_\Omega|\nabla v_\e(x,t)|^2 dx, \  \; a=2, \ \;
h(t)=\frac{2(\mu_2+1)^2}{D_v}\bar{v}_2^2 +\frac{2\beta_3^2}{D_v}\bar{v}_{\infty}^2\int_\Omega u_\e^2\,.$$
Hence, by (\ref{v2}), (\ref{t1}) and Lemma \ref{LA} we  deduce first (\ref{v3}),  and then upon integration (\ref{gradvvv}) from $0$ to $T$ and using the first inequality in (\ref{C1C2}) we obtain also (\ref{v4}). Similarly, (\ref{v44}) results by integrating (\ref{gradvvv}) from $t$ to $t+1$.
Finally, to obtain (\ref{vt44}) we multiply the $v$-equation by $v_{\e,t}$ and integrate with respect to $x$. It yields
\begin{equation}
    \int_\Omega |v_{\e,t}|^2dx +\frac{1}{2}\frac{d}{dt}\int_\Omega |\nabla v_\e|^2dx \leq \frac{\mu_2}{2}\frac{d}{dt}\int_\Omega v_\e^2 dx -\frac{\mu_2}{3}\frac{d}{dt}\int_\Omega v_\e^3dx -\mu_2 a_2\int_\Omega v_e u_\e v_{\e,t}dx
\end{equation}
\begin{align*}
    \int_t^{t+1}\int_\Omega |v_{\e,t}|^2 dxds &+\frac{1}{2}\int_\Omega |\nabla v_\e(x,t+1)|^2 dx \leq \frac{\mu_2}{2} \int_\Omega v^2_\e(x,t+1) dx+\frac{\mu_2}{3} \int_\Omega v^3_\e(x,t) dx \\
    &+ a_2^2 \mu_2^2 \|v_\e\|_{L^\infty(\Omega_T)}^2\int_t^{t+1}\int_\Omega u_\e^2 dx ds + \frac{1}{2}\int_t^{t+1}\int_\Omega |v_{\e,t}|^2 dx ds
\end{align*}
whence (\ref{vt44}) readily follows using the bounds (\ref{v3}) and (\ref{t1})
with respect to time, and using (\ref{vwinf}) along with the Young inequality. 
\end{proof}
\begin{lemma} \label{we} Let $N\geq 1$ and $(u_{\e}, v_{e}, w_{\e})$ be the solution to 
(\ref{ue1})-(\ref{ue3}), then, 
\begin{equation} \label{31-32}
\begin{array}{lcl}
\displaystyle 
  \e \frac{d}{dt}\int_\Omega| \nabla (w_{\varepsilon} -  v_{\epsilon})|^2 
+ \frac{\e}{2}\int_{\Omega} | \Delta (w_{\epsilon}- v_{\epsilon}) |^2 
& + & \displaystyle 
\int_{\Omega} | \nabla (w_{\epsilon}- v_{\epsilon}) |^2 \\
[4mm]
&\leq   & \displaystyle   \e  \left(\int_\Omega |v_{\e, t}|^2 + \int_\Omega |\Delta v_{\epsilon} |^2\right)    .
\end{array}
\end{equation}
\end{lemma}
\begin{proof}
After  subtraction  from  both sides of  $w$-equation  the term 
$ \e ( v_{\e,t} - \Delta v_{\e})$ we arrive at 
\begin{equation} \label{wve}
 \e ( w_{\e} - v_{\e} )_t - \e \Delta ( w_{\e} - v_{\e} )
+ (w_{\e} - v_{\e})
= -\e [  v_{\e,t} - \Delta v_{\e} ].   
\end{equation}
Choosing $-\Delta (w_{\e}- v_{\e})$ as a test function in the previous equation and then applying the Young inequality to the term: 
$$
\left(\sqrt{\e} (v_{\e,t} - \Delta v_{\e})  \right)\left( \sqrt{\e} (\Delta ( w_{\e} - v_{\e})  \right)\,
$$
we arrive at (\ref{31-32}). 
\end{proof}
\begin{lemma} \label{L6} Let $N\geq 1$ and $(u_{\e}, v_{e}, w_{\e})$ be the solution to 
(\ref{ue1})-(\ref{ue3}). Then, there holds  
 \begin{align}\label{wD}
&\sup_{t>0} \int_\Omega| \nabla w_{\varepsilon}|^2 dx \leq C_3 \,,\\ \label{wD2}
&\int_t^{t+1}\int_{\Omega} | \Delta w_{\epsilon}|^2dxds \leq C_4, \quad\mbox{for any}\quad t>0   \\ \label{LapW}
& \int_0^T\int_\Omega |\Delta w_{\e}|^2 dxdt \leq C_w(T)  ,\quad\mbox{for any}\quad T>0 
\end{align}   
where $C_w(T)$ is a positive constant. 
\end{lemma} 
\begin{proof}
From (\ref{31-32}) it follows 
\begin{equation} \label{32}
\begin{array}{lcl}
\displaystyle 
  \frac{d}{dt}\int_\Omega| \nabla (w_{\varepsilon} -  v_{\epsilon})|^2 
+ \frac{1}{2}\int_{\Omega} | \Delta (w_{\epsilon}- v_{\epsilon}) |^2 
& + & \displaystyle 
\frac{1}{\e}\int_{\Omega} | \nabla (w_{\epsilon}- v_{\epsilon}) |^2 \\
[4mm]
&\leq   & \displaystyle  \frac{1 }{2} \left(\int_\Omega |v_{\e, t}|^2 + \int_\Omega |\Delta v_{\epsilon} |^2\right)  \, .  
\end{array}
\end{equation}
 First we notice that due to (\ref{v44}) and (\ref{vt44}), $\int_\Omega |v_{\e, t}|^2 dx + \int_\Omega |\Delta v_{\epsilon} |^2 dx$ is locally integrable with respect to time, and we can apply Lemma \ref{LA} with $\tau=1$ $a=\frac{1}{\e}$ and $b=\frac{1}{2}(\bar{v}_4^\prime + \bar{v}_5^\prime) $ and using (\ref{v3})  it follows 
\begin{align}\label{fu}
 \sup_{t>0}\int_\Omega |\nabla w_\e|^2  &\leq 2\left(\sup_{t>0}\int_\Omega| \nabla (w_{\varepsilon} -  v_{\epsilon})|^2 dx +\sup_{t>0} \int_\Omega |\nabla v_{\epsilon}|^2dx\right) \\ \nonumber
 &\leq \max\{ \|\nabla (w_{\e,0}-v_{\e,0})\|_2 + \frac{1}{2}(\bar{v}_4^\prime + \bar{v}_5^\prime)\,, \frac{1}{2}(\bar{v}_4^\prime + \bar{v}_5^\prime)(\e +2 )\}+\bar{v}_3. 
\end{align}
It remains to note that (\ref{wD2}) is the result of integration (\ref{32}) from $t$ to $t+1$. To show (\ref{LapW})  we integrate (\ref{32}) with respect to time from $0$ to $T$ and use (\ref{v4}) and (\ref{v5}) as well as (\ref{ic}).
\end{proof}

The following lemma is one of the crucial counterparts of our analysis providing an estimate on $\int_\Omega |\nabla w_\e(x,t)|^4dx$. Before its statement, we shall provide some auxiliary results from the literature.
 \begin{itemize}
 \item[(I1)] Bochner's type inequality (see, e.g.\cite{WinklerBound}): For $v\in C^2(\bar{\Omega}) $ there holds
\begin{equation}\label{Boch}
2\nabla v\nabla \Delta v= \Delta |\nabla v|^2-2 |D^2v|^2\,.
\end{equation}
\item[(I2)] Let $u\in C^2(\bar{\Omega})$ satisfy $\frac{\partial u}{\partial \nu}=0$ in $\partial\Omega$ and $\Omega$ be a bounded domain with a regular boundary. Then there holds the following pointwise inequality \cite[Lemma 4.2]{Mizoguchi} 
\begin{equation} \label{Soup}
\frac{\partial|\nabla u|^2}{\partial \nu} \leq K|\nabla u|^2\quad\mbox{on}\quad \partial \Omega
\end{equation}
 where $K=K(\Omega) >0$ is un upper bound on the curvature of $\partial\Omega$.
\item[(I3)] The following inequality is a  well known consequence of the Gagliardo-Nirenberg inequality: for $u\in W^{1,2}(\Omega)$ and any $\eta >0$  there holds
\begin{equation}\label{bb}
\int_{\partial \Omega}u^2 dS \leq \eta \int_\Omega |\nabla u|^2 dx+ C_G(\eta) \left(\int_\Omega u dx\right)^2\,.
\end{equation} 
 \end{itemize}
\begin{lemma} \label{W4}
Let $N\geq 1$. There exist positive constants $\tilde{C}$ and $\tilde{C_1}$ such that the  solution to the $w$-equation satisfies the following inequality 
\begin{equation}
    \frac{d}{dt}\int_\Omega |\nabla w_\e|^4 +\frac{2}{\e}\int_\Omega |\nabla w_\e|^4 \leq \frac{\tilde{C}}{\e}\sup_{t>0} \int_\Omega |\nabla v_\e|^4 + \e   \tilde{C_1} \left(\sup_{t>0}\int_\Omega|\nabla
    w_\e|^2\right)^2.
\end{equation}
\end{lemma}
\begin{proof}
First, observe that using (\ref{Boch}) in the $w$-equation we obtain
\begin{equation}\label{Wt}
\e(|\nabla w_\e|^2)_t= \e\Delta |\nabla w_\e|^2- 2\e|D^2w_\e|^2  + 2\nabla w_\e\cdot\nabla (v_e- w_\e) 
\end{equation}
and then,  integration by parts yields
\begin{align}\nonumber
&\e\frac{1}{2}\frac{d}{dt}\int_\Omega|\nabla w_\e|^4 =\e\int_\Omega|\nabla w_\e|^2(|\nabla w_\e|^2)_t  \\ \nonumber
&=\e\int_\Omega |\nabla w_\e|^2\Delta |\nabla w_\e|^2- 2\e\int_\Omega |\nabla w_\e|^2|D^2w_\e|^2 +2\int_\Omega |\nabla w_\e|^2\nabla w_\e\cdot\nabla (v_\e-w_\e) \\ 
&\leq   -\e\int_\Omega \left\vert\nabla(|\nabla w_\e|^2)\right\vert^2 +\e\int_{\partial\Omega} |\nabla w_\e|^2\frac{\partial |\nabla w_\e|^2}{\partial \nu} \nonumber \\ \label{w44}\
&- 2 \int_\Omega |\nabla w_\e|^4 + 2\int_\Omega |\nabla w_\e|^3|\nabla v_\e|\,.
\end{align}   
Next,  using  the results (I2) and (I3)  to bound the  boundary  integral  we get 
\begin{align*}
&\e\int_{\partial\Omega} |\nabla w_\e|^2\frac{\partial |\nabla w_\e|^2}{\partial \nu}\leq \e K\int_{\partial\Omega} \left(|\nabla w_\e|^2\right)^2 \\
&\leq \e K\eta \int_\Omega
\left\vert\nabla(|\nabla w_\e|^2)\right\vert^2 +  C_G(\eta) \e K \left(\int_\Omega |\nabla w_\e|^2\right)^2
\end{align*}
for any  $\eta >0$.
Hence, choosing above $\eta=\frac{1}{K}$ and  using   the Young inequality to the last term in (\ref{w44}) we arrive at 
$$ \begin{array}{ll} \displaystyle 
\e\frac{1}{2}\frac{d}{dt}\int_\Omega|\nabla w_\e|^4 +& \displaystyle  \e\int_\Omega \left\vert\nabla(|\nabla w_\e|^2)\right\vert^2 +  \int_\Omega |\nabla w_\e|^4  \\
& \displaystyle  \leq  \e\int_\Omega \left\vert\nabla(|\nabla w_\e|^2)\right\vert^2 + \e C_G\left(K^{-1}\right) K\left(\int_\Omega |\nabla w_\e|^2\right)^2  +\frac{3^3}{4^4}\int_\Omega |\nabla v_\e|^4 \, , 
\end{array}$$
whence we deduce the desired result. 
\end{proof}
 Since $\e\in (0,1)$, the following result easily follows from Lemma \ref{W4}. 
\begin{lemma}\label{Cor1} Let $N\geq 1$. The solution to the $w$-equation satisfies
\begin{equation}\label{W44}
 \sup_{t>0}\int_\Omega |\nabla w_\e|^4\leq \max
 \left\{\|\nabla w_{\e,0}\|_4 \,,  k_1\left( \sup_{t>0}\int_\Omega |\nabla v_\e|^4  +  \left(\sup_{t>0}\int_\Omega |\nabla w_\e|^2\right)^2\right) \right\}  
 \end{equation}
 for a positive constant $k_1$. 
\end{lemma}
\begin{proof}
 The proof is a consequence of Lemma \ref{W4}.    
\end{proof}

\noindent
Notice that due to (\ref{wD}) the last summand in (\ref{W44}) is bounded.
\section{A priori estimates for $N\leq 2$} 
To rise the integrability of $u_\e$ we shall use the following versions of the Gagliardo-Nirenberg inequality. 
\begin{itemize}
    \item For $\varphi\in W^{1,2}(\Omega)$ such that $\Omega \subset \R^N$ is a domain with smooth boundary  there holds for $s>0$
\begin{equation} \label{GN1}
 \Vert  \varphi \Vert_4^2\leq C_{GN}\left(\Vert \nabla \varphi \Vert_2^{2\theta}\Vert \varphi \Vert_2^{2(1-\theta)} +\Vert \varphi\Vert_s^2\right)\quad\mbox{with}\;\;\theta =\frac{1}{2}\quad\mbox{for}\;\;N=2\,.
\end{equation}
\item 
For $\psi\in W^{2,2}(\Omega)$ such that $\Omega \subset \R^2$ is a domain with smooth boundary with  outer normal $\nu$ and $\frac{\partial\psi}{\partial\nu}=0$ there holds 
\begin{equation} \label{GN2}
 \Vert \nabla \psi \Vert_4^4\leq C_{Gu}^\prime \Vert \Delta \psi - \psi \Vert_2^2\, \Vert \nabla \psi \Vert_2^2\,.   
\end{equation}
\end{itemize}
Notice that (\ref{GN2}) can be derived from (\ref{GN1}) using the classical theory of elliptic equations with the Neumann boundary condition. Indeed, setting  $\varphi=\partial_i\psi$,  $\theta=\frac{1}{2}$ and $s=2$ in (\ref{GN1}) we readily obtain after transformations 
\begin{align*}
 \Vert  \nabla\psi \Vert_4 &\leq C_{GN}^{\frac{1}{2}}(\Vert D^2 \psi \Vert_2^{\frac{1}{2}}\Vert \nabla\psi \Vert_2^{\frac{1}{2}} +\Vert \psi \Vert_{2,2}^{\frac{1}{2}}\Vert \nabla\psi \Vert_2^{\frac{1}{2}} )  \\
 &\leq 2C_{GN}^{\frac{1}{2}}\Vert \psi \Vert_{2,2}^{\frac{1}{2}}\Vert \nabla\psi \Vert_2^{\frac{1}{2}}\\
 &\leq (C_{GN}^\prime)^{\frac{1}{4}}\Vert \Delta\psi -\psi  \Vert_{2}^{\frac{1}{2}}\Vert \nabla\psi \Vert_2^{\frac{1}{2}}\,.
\end{align*}
For the convenience of the reader we recall also well known $L^p-L^q$  estimates for the parabolic equation with Neumann boundary condition 
\[ z_t-\Delta z +z=\chi_1 \nabla\cdot Q + f\,,\quad z(0)=z_0\in W^{1,q}(\Omega)\,, \quad r>n \]
where $\Omega$ is a bounded domain with smooth boundary, $\chi_1\in \R$ and 
\begin{align}
Q&\in C([0\,,T): W^{1,q_0}(\Omega) )^N) \\
f&\in C([0\,,T):L^q(\Omega))\,.
\end{align} 
For the case $\chi_1=0$,  there is a constant $K_1$ such that 
\begin{equation} \label{zip}
  \Vert \nabla z(\cdot\,,t) \Vert_p\leq K_1 \Vert f\Vert_{C([0\,,T):L^q(\Omega))}\cap C([0,T):C(\bar{\Omega}))
\end{equation}
where 
\[
 \begin{array}{cc}
 p\in [1,\frac{qn}{n-q} )   & ,  q<N \,, \\
 p\in  [1 \,,\infty )   &, q=N \,, \\
 p\in  [1 \,,\infty ]  &, q>N\,
\end{array}
\]
and for the case $\chi_1\neq 0 $ there is a constant $\tilde{K}_2$ such that 
\begin{equation}\label{zinf}
\Vert z(\cdot\,,t) \Vert_\infty \leq \Vert z_0\Vert_\infty + \tilde{K_2}\sup_{s\in [0,t]} \Vert Q(\cdot, s) \Vert_{(L^{q_0}(\Omega))^n} + f_+ \quad\mbox{for}\quad q_0>N
\end{equation}
where  $f_+\geq f(x,t)$ for $(x,t)\in \Omega\times [0,T)$ and $f_+\geq 0$.
The proof of the last inequality, based on \cite[Lemma1.3iv]{WinklerBound} and the order-preserving property of the heat semigroup,  can be found in \cite[pp.397-8]{MiWr2}. 
\begin{lemma} \label{uk}
Let $(u_\e,v_\e,w_\e)$ be a solution to the system (\ref{ue1})-(\ref{ue3}) and $N\leq 2$. 
Then, there exists  a constant $\bar{u}_k$  such that for any $\e\in (0,1)$
\begin{equation}\label{ukk}
\sup_{t>0}\Vert u_\e(\cdot, t)\Vert_k \leq \bar{u}_k\quad\mbox{for any}\quad k\in[2,\infty)\,.
\end{equation}  
\end{lemma}
\begin{proof}
On multiplying the u-equation by $u_\e^{k-1}$,  we obtain for $t>0$
\begin{align}\label{dk}
&\frac{1}{k}\frac{d}{dt}\int_\Omega u_\e^k + (k-1) D_n \int_\Omega  u_\e^{k-2}|\nabla u_\e|^2 \leq\\
&\leq \int_\Omega \chi u_\e (k-1)u_\e^{k-2}|\nabla w_\e||\nabla u_\e| +\int_\Omega \left(\mu_1 u_\e^k- \mu_1 u_\e^{k+1} -a_1\mu_1 u_\e^k v_\e\right) \\ \nonumber
&\leq \int_\Omega \chi (k-1) u_\e^{k-1}|\nabla w_\e||\nabla u_\e| +\int_\Omega \left(\mu_1 u_\e^k- \mu_1 u_\e^{k+1}\right) \\ \nonumber
& \leq \frac{(k-1) D_u}{2} \int_\Omega u_\e^{k-2}|\nabla u_\e|^2 +\frac{\chi^2(k-1)}{2D_u}\int|\nabla w_\e|^2 u_\e^k -\frac{\mu_1}{2}\int_\Omega u_\e^{k+1} +
 \frac{2^{k+1}|\Omega|k^k\mu_1^{k+2}}{\mu^{k}(k+1)^{k+1}} . 
\end{align}
Hence,  using the H\"{o}lder inequality to the component $-\frac{\mu_1}{2}\int_\Omega \ u_\e^{k+1}$ and  multiplying the resultant inequality by $k$ we find
\begin{align} \label{pro}
\frac{d}{dt}\int_\Omega u_\e^k  & +\frac{2D_n (k-1)}{k} \int_\Omega |\nabla \left( u_\e^{\frac{k}{2}}\right)|^2 +\frac{\mu_1 k}{2|\Omega|^{\frac{1}{k}}}\left(\int_\Omega u_\e^k\right)^{1+\frac{1}{k}}\\ \nonumber
&\leq \frac{\chi^2 (k-1)}{D_n}\int_\Omega u_\e^k|\nabla w_\e|^2  +
 \frac{2^{k+1}|\Omega|k^{k+1}\mu_1^{k+2}}{\mu_1^{k}(k+1)^{k+1}} \,.
\end{align}
Next, making use of the Cauchy inequality and the Gagliardo-Nirenberg   inequality (\ref{GN1}) along with the  Young inequality yields 
\begin{align} \nonumber
 \int_\Omega  u_\e^k|\nabla w_\e|^2&\leq \Vert u_\e^{\frac{k}{2}}\Vert^2_4\Vert \nabla w_\e\Vert_4^2\\ \nonumber
 &\leq C_{GN}^2\left(\Vert \nabla u_\e^{\frac{k}{2}}\Vert_2^{2\theta}\Vert u_\e^{\frac{k}{2}}\Vert_2^{2(1-\theta)} +\Vert u_\e^{\frac{k}{2}}\Vert_{\frac{2}{k}}^2\right) \Vert \nabla w_\e\Vert_4^2\\ \nonumber
 &\leq \left(\int_\Omega \left|\nabla \left( u_\e^{\frac{k}{2}}\right)\right|^2\right)^\theta \left(\int_\Omega u_\e^k \right)^{(1-\theta)} C_{GN}^2\Vert \nabla w_\e\Vert_4^2 + \Vert u_\e^{\frac{k}{2}}\Vert_{\frac{2}{k}}^2C_{GN}^2\Vert \nabla w_\e\Vert_4^2 \\ \label{3k}
 &\leq \epsilon \int_\Omega \left|\nabla \left( u_\e^{\frac{k}{2}}\right)\right|^2 +\left(\int_\Omega u_\e^k \right) C_\epsilon \left(C_{GN}^2\Vert \nabla w_\e\Vert_4^2\right)^{\frac{1}{1-\theta}} + \Vert u_\e\Vert_1^2 C_{GN}^2
 \Vert \nabla w_\e\Vert_4^2\,.
\end{align}
In the case $N=2$, using (\ref{GN2}) and the Young inequality, we obtain from (\ref{3k}) setting $\theta=\frac{1}{2}$
\begin{align}\label{lap}
\frac{\chi^2( k-1)}{d_u}\int_\Omega  u_\e^k|\nabla w_\e|^2\leq &  
\epsilon \int_\Omega |\nabla u_\e^{\frac{k}{2}}|^2 \\ \nonumber
&+C_\epsilon C_{GN}^4C_{GN}^\prime \left(\int_\Omega u_\e^k\right)\left(\int_\Omega|\Delta w_\e|^2+ \int_\Omega w_\e^2\right)\Vert \nabla w_\e\Vert_2^2 \\ \nonumber
&+\frac{1}{2}\bar{u}_1^4C_{GN}^4+\frac{1}{2} \left( \int_\Omega|\Delta w_\e|^2 +\int_\Omega w_\e^2\right) + \frac{C_{GN}C_2^\prime}{2}^\prime\Vert \nabla w_\e\Vert_2\,.
\end{align}
Notice that by (\ref{wD}) we have a uniform bound on the term $\sup_{t>0}\Vert \nabla w_\e(\cdot,t)\Vert_2^2$ and by (\ref{wD2}) on $\Vert \Delta w_\e\Vert_2^2\in L^1_{loc}(0\,,\infty)$.  Hence, choosing $\epsilon=\frac{2 d_u^2}{\chi^2}$ from (\ref{pro}) we conclude that 
\begin{equation}\label{llap}
\frac{d}{dt}\int_\Omega u_\e^k + \frac{\mu_1 k}{2|\Omega|^{\frac{1}{k}}}\left(\int_\Omega u^k\right)^{1+\frac{1}{k}}\leq C_1 \left(\int_\Omega u_\e^k\right)\Vert \Delta w_\e\Vert_2^2+ C_2\Vert \Delta w_\e\Vert_2^2 +C_3
\end{equation}
where $C_1, C_2, C_3$ are some positive constants. Now, in view of (\ref{wD2}) we   apply Lemma \ref{LB} with  $$\xi(t)=\int_\Omega u_\e(x,t)^kdx\,, a_1=\frac{\alpha_2 k}{2|\Omega|^{\frac{1}{k}}}\,,
\quad \sigma=\frac{1}{k}\,,\quad g(t)=C_1 \Vert \Delta w_\e(\cdot ,t)\Vert_2^2\,,$$
\[h(t)=C_2 \Vert \Delta w_\e(\cdot ,t)\Vert_2^2 +C_3\,,
\]
 whence  (\ref{ukk})  follows. Notice that,  the right hand side term in 
 the $v$-equation is bounded for  $N=1$, therefore, from (\ref{zip}) and Lemma \ref{Cor1},  there exist  constants $C_v, C_w$ such that 
\[\sup_{t>0}\|\nabla v_\e(\cdot,t)\|_4<C_v\,,
 \; \; \; \; \;   
 \sup_{t>0} \|\nabla w_\e(\cdot,t)\|_4<C_w\,.\]
Owing this bound  we readily obtain  from (\ref{pro}) and (\ref{3k}) that there exist  constants $A_1$ and $A_2$ such that 
\begin{equation}\label{llap1}
\frac{d}{dt}\int_\Omega u_\e^k + \frac{\mu_1 k}{2|\Omega|^{\frac{1}{k}}}\left(\int_\Omega u^k\right)^{1+\frac{1}{k}}\leq A_1 \left(\int_\Omega u_\e^k\right)  +A_2    
\end{equation}
and again Lemma \ref{LB} can be applied to conclude. 
 \end{proof}

\begin{lemma}\label{Cor2}
 For space dimension $N\leq 2$ there is a constant $\bar{u}_\infty$ such that 
 \begin{equation} \label{u12}
     \sup_{t>0}\Vert u_\e\Vert_\infty \leq \bar{u}_\infty\,.
 \end{equation}
\end{lemma}
\begin{proof}
Notice first that in light of (\ref{vwinf}) and Lemma \ref{uk} the reaction part of the $v$-equation satisfies
\[\sup_{t>0} \Vert \mu_2v_e(1-v_\e-a_2u_\e)\Vert_k <\infty \, \]
for any $k\geq 1$. Next,  using (\ref{zip}) for $N\leq 2$  we infer that there is a constant $K_4$ such that 
\begin{equation}\label{K4}
\sup_{t>0}\Vert \nabla v_\e(\cdot\,,t) \Vert_4\leq K_4\,.
\end{equation}
It then follows from (\ref{v3}) and Lemma \ref{Cor1} that there is a constant $\tilde{K}_4$ such that 
\begin{equation}\label{we4}
 \sup_{t>0}\Vert \nabla w_\e(\cdot\,,t) \Vert_4\leq \tilde{K}_4\,.   
\end{equation}
By Lemma \ref{uk}, (\ref{we4}) and the H\"older inequality we infer  that 
\[\sup_{t>0}\Vert u_\e(\cdot\,,t)\nabla w_\e(\cdot\,,t) \Vert_p <\infty\quad\mbox{for some}\quad p\in (3,4)\,.\]
Now  (\ref{u12}) follows readily from (\ref{zinf}) with $f_+=\frac{\mu_1}{4a_1}$.
\end{proof}
\section{Convergence} 
Taking into account the estimates derived in the previous section, we proceed to show the convergence of solutions as $\e\rightarrow 0$ first for the case $N\leq 2$ (in the next section) where there are regular solutions and the compatibility condition (\ref{ic}) is satisfied. Then, in the following section we use the compactness argument to show the convergence in the case $N\geq 3$.
\subsection{Convergence of full sequence as  $\e \rightarrow 0$ for $N\leq 2$.}
In this subsection we show that the solution $(u_\e, v_\e)$ converges, as  $\e\rightarrow 0$, to the solution $(u\,,v)$ of the limit problem 
\begin{align} \label{u}
   u_t - d_u\Delta u &= \chi \nabla \cdot u \nabla v +\mu_1(1-u-a_1v)\,,\\ \label{v}
   v_t -d_v\Delta v &= \mu_2(1-v-a_2v)\,.
\end{align}
 We recall that for the space dimension $N\leq 2$ system (\ref{u})-(\ref{v})  has a global in time $L^\infty$-bounded classical solution. If 
 \begin{equation}\label{icreg2}
 u_0\,,v_0\in W^{1,q}(\Omega)\,, 
 \;  \mbox{ for } q>N
 \end{equation}
 are nonnegative functions, then there exists a global in time unique classical solution to (\ref{u})-(\ref{v}) (cf. \cite{Yagi} and \cite[Sec. 2]{Le} or \cite{MiWr}) such that for any $T>0$ 
\[(u,v)\in (C([0\,,T]:W^{1,r}(\Omega))\cap 
C^{2,1}(\bar{\Omega}\times (0\,,T)))^2\, , \] 
and moreover, there exist constants $\bar{U}_\infty\,,\bar{V}_\infty$ such that 
\begin{equation}\label{uvM}
\sup_{t\in [0\,,\infty)}\Vert u(\cdot,t)\Vert_{\infty}\leq \bar{U}_\infty\,,\quad 
\sup_{t\in [0\,,\infty)}\Vert v(\cdot,t)\Vert_{\infty}\leq \bar{V}_\infty\, , 
\end{equation}   
and if in addition $v_0\in W^{2,2}(\Omega)$ then for any $T>0$ there exists $C_v(T)$
\begin{equation}\label{vdelta}
  \sup_{t\in [0\,,T]}  \|\Delta v(\cdot\,,t) \|_2\leq C_v(T). 
\end{equation}
\begin{lemma}\label{UVW} 
Let $N \leq2$, (\ref{ic}) be satisfied, and in addition to (\ref{icreg}) and (\ref{icreg2}), $v_0\in W^{2,2}(\Omega)$ with $T<\infty$. Then 
for some positive constant $\tilde{C}(T)$  depending on $T$
\begin{equation}\label{ctilde}
\left( 
\int_{\Omega} U_{\epsilon}^2dx +\int_{\Omega} V_{\epsilon}^2dx +\int_{\Omega} |\nabla V_{\epsilon}|^2dx +
 \int_{\Omega} |\nabla W_{\epsilon}|^2dx
\right) \leq \e \tilde{C}(T),     
\end{equation}
where  $$U_{\varepsilon}= u - u_{\varepsilon} , 
\; \; 
V_{\varepsilon}= v- v_{\varepsilon} \; \; W_{\varepsilon}= v-w_{\varepsilon}\,.
$$

\end{lemma}
\begin{proof}
After subtracting,  we obtain 
\begin{equation}\label{Uru}
    U_{\varepsilon t}- d_u\Delta U_{\varepsilon } =  
\nabla \cdot \left[  \chi\left(  U_{\e} \nabla  w_\e  + u \nabla  W_{\e} \right) \right]
+ \mu_1 U_{\e}( 1- u -a_1 v )-
\mu_1 u_{\e}( U_{\e}+a_1 V_{\e} )\,.
\end{equation}
Then,  on multiplying  by $U_{\varepsilon}$ and integrating  by parts we use the Young inequality to get, after some computations: 
  \begin{align}\label{Us1}
 \frac{1}{2} \frac{d}{dt} 
 \int_{\Omega} U_{\varepsilon}^2
 + \frac{d_u}{2} 
 \int_{\Omega} |\nabla U_{\varepsilon}|^2
 \leq & \displaystyle   \frac{\chi^2}{d_u} \left(
 \int_{\Omega} U_{\varepsilon}^2 |\nabla w_{\varepsilon} |^2
 +  \int_{\Omega} u^2 |\nabla W_{\varepsilon} |^2\right) 
\\ \nonumber
&+
 \mu_1 \int_{\Omega} U_{\e}^2  -  \mu_1 a_1 \int_{\Omega} U_\e V_\e u_\e dx \,.
 \end{align}
 In view of (\ref{we4}) we get 
$$ \int_{\Omega} U_{\varepsilon}^2 |\nabla w_{\varepsilon} |^2
 \leq \| U_{\varepsilon} \|_{L^4((\Omega)}^2 \| \nabla w_{\varepsilon} \|_{L^{\infty}(0,T: L^4(\Omega))}^2
 \leq K_4^2  \|  U_{\varepsilon} \|^2_{L^4(\Omega) }
 $$
and by  the Gagliardo-Nirenberg inequality (\ref{GN1}) with $s=2$ and the Young inequality we have 
 \begin{align*}
 \frac{\chi^2K_4^2}{d_u}\| U_{\varepsilon} \|^2_{L^4(\Omega) } &\leq \frac{\chi^2K_4^2C_{GN}}{d_u}\left(\|  U_{\varepsilon} \|_{2 } \| \nabla  U_{\varepsilon} \|_{2 } + \| U_{\varepsilon} \|_{2 }^2\right) \\
 &\leq \frac{d_u}{4} \|\nabla U_{\varepsilon} \|^2 + \left(\frac{\chi^4K_4^4C_{GN}^2}{d_u^3} + \frac{\chi^2K_4^2C_{GN}}{d_u} \right) \|  U_{\varepsilon} \|_{2 }^2 \\
 &\leq \frac{d_u}{4} \|\nabla U_{\varepsilon} \|^2 + \gamma_1\|  U_{\varepsilon} \|_{2 }^2\,.
 \end{align*}
  Coming back to (\ref{Us1}) we obtain
 $$ 
\frac{\chi^2}{d_u}\int_{\Omega} u^2 |\nabla W_{\varepsilon} |^2
\leq 
\frac{\chi^2}{d_u}\sup_{t>0}\| u\|_\infty^2 
\|\nabla W_{\varepsilon} \|^2_{2}\leq \frac{\chi^2c}{d_u}\|\nabla W_{\varepsilon} \|^2_{2}\leq \frac{\gamma_2}{2} \|\nabla W_{\varepsilon} \|^2_{2} 
$$
where 
\begin{equation}\label{cm}
 c=\max\{\sup_{t>0}\| u\|_\infty\,, \sup_{t>0}\| u_\e\|_\infty\,, \sup_{t>0}\| v\|_\infty\,, \sup_{t>0}\| v_\e\|_\infty \} 
 \end{equation}
and $\gamma_2:=\frac{2\chi^2c}{d_u}$. It follows that 
\begin{align*}
\frac{1}{2} \frac{d}{dt} 
 \int_{\Omega} U_{\varepsilon}^2
 +  \frac{d_u}{4} 
 \int_{\Omega} |\nabla U_{\varepsilon}|^2\leq \left(\gamma_1   + \mu_1 \left((1+\frac{a_1c}{2}\right) \right) 
 \int_{\Omega} U_{\varepsilon}^2  
 + 
 \frac{\mu_1 a_1}{2} \int_{\Omega} V_{\e}^2  + 
 \frac{\gamma_2}{2}
 \int_{\Omega}   |\nabla W_{\varepsilon}|\,. 
    \end{align*}
Hence,  
\begin{equation} \label{Ue}  
 \frac{d}{dt} \int_{\Omega} U_{\varepsilon}^2
 + \frac{d_u}{2} 
 \int_{\Omega} |\nabla U_{\varepsilon}|^2
 \leq c_1
 \int_{\Omega} U_{\varepsilon}^2 + c_2 \int_{\Omega} V_{\varepsilon}^2 
 + \gamma_2 \int_{\Omega}   |\nabla W_{\varepsilon} |^2 
  \end{equation} 
with $ c_1=2\left(\gamma_1   + \mu_1 (1+\frac{a_1c}{2}) \right)$ and  $ c_2= \mu_1 a_1$. Notice that  $W_{\varepsilon}= v- w_{\epsilon} $ satisfies 
\begin{align*}
 \varepsilon W_{\epsilon t} - \varepsilon \Delta W_{\varepsilon}
+ W_{\varepsilon}
&= V_{\epsilon}+\varepsilon [v_t - \Delta v]   \\
&= V_{\epsilon}+\varepsilon \mu_2[v(1- a_{2} u-v)]+ \e (1-d_v) \Delta v . 
\end{align*}
Next, we multiply by $- \Delta W_{\varepsilon}$ and integrate by parts to obtain  
\begin{align*}
&\frac{d}{dt} \frac{\e}{2} 
\int_{\Omega} |\nabla W_{\e}|^2 + \e\int_{\Omega} |\Delta W_{\e}|^2 dx
+ \int_{\Omega} | \nabla W_{\e}|^2 \\
&= \int_{\Omega} \nabla V_{\e} \nabla W_{\e} -  \e \mu_2 \int_{\Omega } \Delta W_{\e}[v(1- a_{2} u-v)] dx
+ \e (d_v-1)  \int_{\Omega} 
\Delta W_{\e} \Delta v.    
\end{align*}
Since 
$$ \int_{\Omega} \nabla V_{\e} \nabla W_{\e} \leq 
\frac{1}{2} \int_{\Omega} |\nabla V_{\e} |^2 + \frac{1}{2} \int_{\Omega} |\nabla W_{\e} |^2 $$
and by the Young inequality 
$$
 \e \mu_2 \int_{\Omega } \Delta W_{\e} [ v(1- a_{2} u-v)]  
 \leq \frac{\e}{4} \int_{\Omega} |\Delta W_{\e}|^2 +
 \e|\Omega|\beta_1^2
$$
with  
$\beta_1= 2\max\{1,\mu_2\} \max\{1,a_2\}(\max\{1,c\})^2$. Using (\ref{vdelta}) we find 
$$
\e (d_v-1)  \int_{\Omega} 
\Delta w_{\e} \Delta v   
\leq 
\frac{\e}{4} \int_{\Omega} |
\Delta w_{\e} |^2
+ \e (1- d_v)^2 \int_{\Omega} |
\Delta v |^2\leq \frac{\e}{4} \int_{\Omega} |
\Delta w_{\e} |^2 +\e(1+d_v^2)C_v(T)
$$
and finally for $t\leq T$ we obtain 
 \begin{align} 
 \label{We}
  \frac{d}{dt} \frac{\e}{2} 
 \int_{\Omega} |\nabla W_{\e}|^2 +
\frac{\e}{2} \int_{\Omega} |\Delta W_{\e}|^2  
+ \frac{1}{2} \int_{\Omega} |\nabla W_{\e}|^2 \leq \frac{1}{2} \int_{\Omega} |\nabla V_{\e} |^2  +  \\ \nonumber
\e(\beta_1^{2}|\Omega|+ (1+d_v^2)C_v(T)):=\e\beta_2(T).
\end{align}

\noindent
We also notice that  $V_{\varepsilon} = v- v_{\e}$ satisfies: 
\begin{equation}\label{Vru}
V_{\epsilon t} -  d_v \Delta V_{\varepsilon}
=\mu_2  V_{\epsilon} (1- a_2 u -v) - \mu_2  v_{\e}( a_2U_{\e} + V_{\e}   ). 
\end{equation}
On multiplying  this equation by $(V_{\e}-\Delta V_\e)$ and then  integrating  by parts we find using  the Young inequality  
\begin{equation} \label{Ve}
\frac{d}{dt} 
\left(\int_{\Omega} V_{\epsilon}^2 + 
\int_{\Omega} |\nabla V_{\e}|^2\right) +d_v\int_\Omega |\Delta V_\e|^2 \leq 
2\beta_1(2 + \beta_1 d_v)\int V_{\e}^2  +
2\beta_1(\beta_1 +d_v)\int U_{\e}^2 
\end{equation} 
Next,  we multiply the equation  (\ref{We}) by $2\gamma_2$ and add to equations (\ref{Ue}) and (\ref{Ve}) to obtain 
\begin{align*}
 & \frac{d}{dt} \left( 
\int_{\Omega} U_{\epsilon}^2 +\int_{\Omega} V_{\epsilon}^2 +\int_{\Omega} |\nabla V_{\e}|^2 
+  \e \gamma_2
\int_{\Omega} |\nabla W_{\epsilon} |^2 
\right) \\
&\leq \alpha_1 \left(\int U_{\e}^2 + \int_\Omega V_{\e}^2 +\int_{\Omega} |\nabla V_{\e}|^2\right) 
+ \e 2\beta_2(T)\gamma_2   
 \end{align*}
 where $\alpha_1=\max\{\gamma_2\,,c_1\,,c_2\,,2\beta_1(\max\{2,\beta_1\} (1+d_v)\}$
In view of (\ref{ic}) we have  
$$  U_\e(\cdot,0)=V_\e(\cdot,0)=W_\e(\cdot,0)=0
$$
and therefore  by comparison with O.D.E we obtain that for $t\in [0\,,T] $
\begin{align*}
 &\left( 
\int_{\Omega} U_{\epsilon}(\cdot, t)^2 +\int_{\Omega} V_{\epsilon}(\cdot, t)^2 +\int_{\Omega} |\nabla V_{\e}(\cdot, t)|^2 
+ \e \gamma_2
\int_{\Omega} |\nabla W_{\epsilon}(\cdot, t) |^2
\right) \\
& \leq \e \frac{2\beta_2(T) \gamma_2 (e^{\alpha_1 T } -1)}{\alpha_1} :=\e C(T) .   
\end{align*}
From (\ref{We}) we infer that 
 \begin{equation} 
 \label{We1}
  \e \frac{d}{dt}
 \int_{\Omega} |\nabla W_{\e}|^2 +  \int_{\Omega} |\nabla W_{\e}|^2 \leq \e \left(C(T) + 2\beta_2(T) \right) \,
\end{equation} 
and dividing it by $\e$ we obtain  
\begin{equation} 
 \label{We11}
 \frac{d}{dt}
 \int_{\Omega} |\nabla W_{\e}|^2 +  \frac{1}{\e}\int_{\Omega} |\nabla W_{\e}|^2 \leq \left(C(T) + 2\beta_2(T) \right) \,.
\end{equation}Hence, we deduce by comparison with O.D.E. using (\ref{ic}) that  
\begin{equation} \label{We2}
\int_{\Omega} |\nabla W_{\e}(\cdot, t)|^2\leq \e \left(C(T) + 2\beta_2(T) \right)\quad\mbox{for}\quad t\in [0\,,T] 
\end{equation}
and (\ref{ctilde}) follows with $\tilde{C}(T) =2(C(T)+\beta_2(T))$. 
\end{proof} 

We are in a position to formulate one of our main results.
\begin{theorem} \label{L22}
 Let $N \leq2$, (\ref{ic}) be satisfied, and in addition to (\ref{icreg}) and (\ref{icreg2}),$v_0\in W^{2,2}(\Omega)$ with $T<\infty$. Then the solution $(u_\e\,,v_\e)$ converges in the following sense to solution $(u\,,v)$ to (\ref{u})-(\ref{v}) 
 \begin{align} \label{p1}
 \sup_{t\in[0,T]} \|u(\cdot, t)-u_\e(\cdot, t)\|_2\leq \e \tilde{C}(T)  \\ \label{p2}
 \sup_{t\in[0,T]} \|v(\cdot, t)-v_\e(\cdot, t)\|_{W^{1,2}(\Omega)}\leq \e \tilde{C}(T)  \\ \label{p3}
 \sup_{t\in[0,T]} \|\nabla v(\cdot, t)-\nabla w_\e(\cdot, t)\|_2\leq \e \tilde{C}(T)  . 
 \end{align}
Moreover,  
\begin{equation}\label{p4}
\| u- u_\e\|_{L^2(0\,,T:W^{1,2}(\Omega)} \leq \e T\tilde{C}(T)
\end{equation}   
and  for any sequence $\e=\e_j\searrow 0$
\begin{align}\label{p5}
    u_{\e,t} &\rightarrow u_t \quad\mbox{weakly in }\quad L^2(0\,,T:L^{2}(\Omega) ),\\
    \label{p6}
    v_{\e,t} &\rightarrow v_t \quad\mbox{in}\quad L^2(0\,,T:L^{2}(\Omega) ).
\end{align}
\end{theorem}
\begin{proof}
Notice that (\ref{p1})-\ref{p3}) stem immediately from Lemma \ref{UVW} while (\ref{p4}) results from (\ref{Ue}) integrating with respect to $t$ and using (\ref{p1}) -(\ref{p3}).
To get (\ref{p5}) we choose any sequence $\e=\e_j\searrow 0$ and multiply (\ref{Uru}) by a test function $\varphi\in L^2(0\,,T:L^2(\Omega))$ and after integration using the Cauchy inequality we obtain
\begin{align}\label{slu}
&\left\vert\int_0^T \int_\Omega (u_t-u_{\e,t})\varphi dxdt \right\vert \leq d_u\Vert U_\e\Vert_{L^2(0\,,T:W^{1,2}(\Omega)}\Vert \varphi\Vert_{L^2(0\,,T:W^{1,2}(\Omega)} \\ \label{slu2}
&+ \chi\int_0^T\int_\Omega U_\e \nabla w_\e \nabla\varphi + \chi Tc\Vert W_\e\Vert_{L^2(0\,,T:W^{1,2}(\Omega)}\Vert \varphi\Vert_{L^2(0\,,T:W^{1,2}(\Omega)} \\
&+C_0\left(\Vert U_\e\Vert_{L^2(0\,,T:L^{2}(\Omega)} +\Vert V_\e\Vert_{L^2(0\,,T:L^{2}(\Omega)} \right) \Vert \varphi\Vert_{L^2(0\,,T:W^{1,2}(\Omega)}
\end{align}
where $C_0$ is a constant depending on $\mu_1\,,a_1$ and $c$ from (\ref{cm}). In order to pass to the limit in the first term in (\ref{slu2}) we notice that due to (\ref{p1}) and $L^\infty(\Omega)$ bound on $ U_{\e}$ we infer that from any subsequence of $\{\e_j\}_{j\geq 0}$ we can choose a subsequence $\{\tilde{\e}_j\}_{j\geq 0}$ such that 
\begin{equation}\label{star}
  U_{\tilde{\e}} \rightarrow 0 \quad\mbox{weakly $\star$ in }\quad L^\infty(0\,,T:L^{\infty}(\Omega) )
\end{equation}
and owing to (\ref{p3}) we deduce that 
\begin{equation}\label{star1}
  \nabla w_{\tilde{\e}}\nabla\varphi \rightarrow v\nabla\varphi \quad\mbox{in }\quad L^1(0\,,T:L^{1}(\Omega))\,.
\end{equation}
 Since for any subsequence the limit is identical, we infer the convergence of the full sequence  and finally making use of (\ref{p1})-(\ref{p3}) we conclude that (\ref{p5}) is true, while (\ref{p6}) readily follows from (\ref{Vru}) after multiplying by $V_{\e,t}$ and using (\ref{p1})-(\ref{p2}). 

\end{proof}
In light of (\ref{p1})-(\ref{p6}) and (\ref{star})-\ref{star1})  we are in a position to let $\e\rightarrow 0$ in the following weak formulation 
\begin{align*}
     \int_{0}^T \int_{\Omega} u_{\e,t} \varphi &= -d_u \int_{0}^T \int_{\Omega} \nabla u_\e \nabla \varphi
- \chi \int_{0}^T \int_{\Omega} u_\e \nabla w_\e \nabla \varphi \\
&+  \mu_1\int_{0}^T \int_{\Omega} (u_\e (1- u_\e - a_1 v_\e) ) \varphi\\
 \int_{0}^T \int_{\Omega} v_{\e,t } \psi 
=&   \int_{0}^T \int_{\Omega} \nabla v_\e \nabla \psi +
\mu_2\int_{0}^T \int_{\Omega} ( v_\e (1- v_\e - a_2 u_\e) ) \psi
\end{align*}
 for all $\varphi\,,\psi\in L^2(0\,, T: W^{1,2}(\Omega))$
and as a result we obtain 
\begin{align}\label{uw1}
     \int_{0}^T \int_{\Omega} u_{t} \varphi &= -d_u \int_{0}^T \int_{\Omega} \nabla u \nabla \varphi
- \chi \int_{0}^T \int_{\Omega} u \nabla v \nabla \varphi \\ \nonumber
& +  \mu_1\int_{0}^T \int_{\Omega} (u( 1- u - a_1 v) ) \varphi\\ \label{vw2}
 \int_{0}^T \int_{\Omega} v_{t} \psi 
=&   \int_{0}^T \int_{\Omega} \nabla v \nabla \psi +
\mu_2\int_{0}^T \int_{\Omega} (v ( 1- v - a_2 u) ) \psi
\end{align}
 for all $\varphi\,,\psi\in L^2(0\,, T: W^{1,2}(\Omega))$.

\bigskip
\subsection{Uniqueness of the solutions} 
There arises a question whether a $L^2$ weak solution solution $(u,v)$ to (\ref{uw1})-(\ref{vw2})  is uniquely determined. Let us define a class of functions 
\begin{align*}
 \mathcal{Y}=\{(\sigma\,,\eta)\in \left[ L^2(0\,, T: W^{1,2}(\Omega))\cap W^{1,2} (0\,, T: L^{2}(\Omega))
\cap  L^\infty(\Omega_T)\right]^2\,, \\ \nabla \eta \in  L^\infty(0,T:L^4(\Omega))\}   .
\end{align*}
\begin{theorem} \label{uniq}
 For $N\leq 2$ in the class $\mathcal{Y}$ there is a unique weak solution to (\ref{uw1})-(\ref{vw2})  starting from the initial data $(u_0\,,v_0)\in (L^\infty(\Omega)\times W^{1,4}(\Omega))^2$.
\end{theorem}
\begin{proof} We present the proof for $N=2$, since the proof for $N=1$ is similar, where only  the exponents used in the Gagliardo-Nirenberg inequality have to be modified, we omit the details for this case.

Assume that there are two solutions in the class $\mathcal{Y}$; $(u_1,v_1)$, $(u_2,v_2)$ with the same initial condition and let 
$$ U=u_1-u_2\,, V=v_1-v_2\,.$$
Then,  after the subtraction of suitable equations, integration with respect to $x\in\Omega$ and making use of the Young inequality we obtain
\begin{align*}
\frac{1}{2}\frac{d}{dt} \int_\Omega U^2+d_u\int_\Omega |\nabla U|^2\leq \chi\int_\Omega (U\nabla v_2 + u_1\nabla V)\nabla U + C_1\left(\int_\Omega U^2 +\int_\Omega V^2\right)    
\end{align*}
and 
$$
\frac{1}{2}\frac{d}{dt}\left( \int_\Omega V^2+\int_\Omega |\nabla V|^2\right)+d_v\int_{\Omega} |\nabla V|^2\leq  \mu_2 \int_{\Omega} |\nabla V|^2 +C_2\left(\int_\Omega U^2 +\int_\Omega V^2\right)    
$$
where $C_1$ and $C_2$ are constants depending on 
$a_1\,,a_2\,,\mu_1\,,\mu_2\,,\|u\|_{L^\infty(\Omega_T)}\,,\|v\|_{L^\infty(\Omega_T)}$.
Notice that  by the Young inequality,  we obtain
\begin{align*}
 \chi\int_\Omega (U\nabla v_2 + u_1\nabla V)\nabla U &\leq \frac{d_u}{4} \int_\Omega |\nabla U|^2 + \frac{\chi^2}{d_u}\|u_1\|_{L^\infty(\Omega_T)}^2\int_\Omega |\nabla V |^2 \\
 &+\frac{d_u}{4} \int_\Omega |\nabla U|^2 +\frac{\chi^2}{d_u} \int_\Omega U^2|\nabla v_2|^2 
  \end{align*}
  and by the Gagliardo-Nirennberg inequality, (\ref{GN1}) and the Young inequality, we have
 $$ \begin{array}{lcl}  \displaystyle 
 \frac{\chi^2}{d_u} \int_\Omega U^2|\nabla v_2|^2 & \leq
 & 
 \frac{\chi^2C_{GN}}{d_u}\left(\|\nabla U\|_2 \|U\|_2 + \|U\|_2^2\right)\|\nabla v_2\|_4^2\\ [2mm] 
 & \leq & \displaystyle 
 \frac{d_u}{4} \int_\Omega |\nabla U|^2 + \left(  \frac{\chi^4C_{GN}^2}{d_u^2}  \|\nabla v_2\|_4^4 +  \|\nabla v_2\|_4^2 \right) \int_\Omega U^2\,.
 \end{array}$$
Finally, by definition of $\mathcal{Y}$, we infer that there is a constant $C_3$ such that 
$$
\frac{1}{2}\frac{d}{dt} \left( \int_\Omega U^2 + \int_\Omega V^2 +\int_\Omega |\nabla V|^2\right) +\frac{d_u}{4} \int_\Omega |\nabla U|^2\leq 
C_3\left( \int_\Omega U^2 + \int_\Omega V^2 +\int_\Omega |\nabla V|^2\right)
$$
whence the result readily follows by Gronwall$^{\prime}$s Lemma.
\end{proof}

\subsection{Convergence by using the compactness argument} 
Due to the lack of suitable estimates for $N>2$ testing the $u$-equation by $u_\e$ is not allowed, and we proceed in a way similar to that in \cite{Wink1} and \cite{Lankeit} arriving at the concept of weak solution close to that in the aforementioned paper where the Keller-Segel system with logistic growth was studied. In this section $\e=\e_j\searrow 0$ is assumed to be a sequence of numbers. The case $N\leq 2$ will be considered at the end of this section.
\begin{lemma} \label{lnu} 
Suppose that $N\geq 1$ and $T>0$, then 
\begin{equation}\label{lnn} 
\int_0^T \int_{\Omega} |\ln (1+ u_\e) |^2+\int_0^T \int_{\Omega} |\nabla \ln(1+ u_\e) |^2 \leq c\,.
\end{equation}
\end{lemma}
\begin{proof} 
First, we infer from the obvious inequality
\begin{equation} \label{ln}
\ln{(1+y)}\leq y
\end{equation}
and (\ref{C1C2}) that the first summand in (\ref{lnn}) is bounded. 
On multiplying the $u$-equation  by $\frac{1}{1+u_{\e}}$  we  obtain 
\begin{align*}
\frac{d}{dt}\int_\Omega \ln (1+u_\e) dx =& d_u\int_\Omega \frac{|\nabla u_\e|^2}{(1+u_\e)^2} dx + \int_\Omega 
\frac{u_\e}{(1+u_\e)^2}\nabla w_\e \nabla u_\e dx \\
&+ \mu_1\int_\Omega \frac{u_\e}{1+u_\e} dx- \mu_1\int_\Omega \frac{u_\e^2}{1+u_\e} dx-\mu_1 a_1 \int_\Omega \frac{u_\e v_\e}{1+u_\e}dx\,.    
\end{align*}
Making use of the Young inequality to the term
\[ \int_\Omega 
\frac{u_\e}{(1+u_\e)^2}\nabla w_\e \nabla u_\e dx=\int_\Omega 
\left(\frac{u_\e}{1+u_\e}\right)\left(\frac{\nabla u_\e}{1+u_\e}\right) \nabla w_\e  dx,\]
it results in 
\begin{align*}
&d_u\int_0^T\int_\Omega \frac{|\nabla u_\e|^2}{(1+u_\e)^2} dx\leq \frac{d_u}{2} \int_0^T\int_\Omega \frac{|\nabla u_\e|^2}{(1+u_\e)^2} dx+ \frac{1}{2d_u}\int_0^T\int_\Omega |\nabla w_\e|^2 dx\\
&+\int_\Omega \ln (1+u(T,\cdot))dx -\mu_1\int_0^T\int_\Omega\frac{u_\e}{1+u_\e}dx + \mu_1\int_0^T\int_\Omega\frac{u_\e^2}{1+u_\e}dx + \mu_1 a_1 \int_0^T\int_\Omega\frac{u_\e v_\e}{1+u_\e}dx. 
\end{align*}
Then we use (\ref{ln}) along with (\ref{M1}) as well as (\ref{wD}) to conclude that  there is a constant $c(T)$ such that 
 $$ \int_0^T \int_{\Omega} |\nabla \log(1+ u_\e) |^2 dx =\int_0^T\int_\Omega \frac{|\nabla u_\e|^2}{(1+u_\e)^2} dx \leq c(T)\,.$$
 Hence, (\ref{lnn}) follows. 
 \end{proof}
 It allows us to deduce additional information about the sequence $\{u_\e\}$ itself. 
\begin{lemma}
 Suppose that $N\geq 1$ and $T>0$, then there exist  constants $\tilde{C}_1\,, \tilde{C}_2\,,\tilde{C}_3(T)$ such that 
\begin{align}\label{L43} 
\|u_\e\|_{L^{\frac{4}{3}}(0,T:W^{1,\frac{4}{3}}(\Omega))} &\leq \tilde{C}_1(T)   ,  \\ \label{Lt43}
\|u_{\e,t}\|_{L^1(0,T:(W^{2,\infty}(\Omega))^*) } &\leq \tilde{C}_2(T),\\ \label{ln43} 
\int_0^T \int_\Omega u_\e^2\ln(1+u_\e)dx dt &\leq \tilde{C}_3(T).
\end{align}
\end{lemma}
\begin{proof}
The proof of (\ref{L43}) is based on Lemma \ref{lnu}, the first inequality in (\ref{C1C2}) and the H\"{o}lder inequality, for details we refer the reader to \cite[Lemma 3.6]{Lankeit}.
The proof of (\ref{Lt43}) is equivalent in fact to \cite[Lemma 3.7]{Lankeit} which we present for the reader's convenience.
For any $\varphi \in W^{2,\infty}(\Omega)$ such that $\|\varphi\|_{W^{2,\infty}(\Omega)}\leq 1$ we obtain
\begin{align*}
 &\int_0^T \left|\int_\Omega u_{\e,t}\varphi \right|dxdt\leq d_u\int_0^T \int_\Omega |u_\e \Delta \varphi|dxdt \\
 &+ \chi\int_0^T \int_\Omega|u_\e\nabla w_\e \nabla\varphi| dxdt +\mu_1\int_0^T \int_\Omega u_\e|\varphi| dxdt \\
 &+\mu_1\int_0^T \int_\Omega u_\e^2|\varphi|dxdt 
 + \mu_1a_1\int_0^T\int_\Omega u_\e v_\e|\varphi| dxdt\\
 &\leq T\bar{u}_1(d_u +\mu_1(1+a_1\bar{v}_\infty))+ \left(\mu_1+\frac{\chi_2^2}{2}\right)\int_0^T\int_\Omega u_\e^2dxdt +
 \int_0^T\int_\Omega |\nabla w_\e|^2dxdt
\end{align*}
hence (\ref{Lt43}) follows in view of (\ref{M1}),  (\ref{vwinf})  and the first inequality in  (\ref{C1C2}). The bound (\ref{ln43}) results from multiplying the $u$-equation by $\ln (1+u_\e)$. We omit the proof, since it follows lines from that of \cite[Lemma 3.5]{Lankeit} and makes use of (\ref{LapW}).
\end{proof}
We are in a position to indicate the convergence needed to pass to the limit in the weak formulation.
\begin{theorem}\label{Lw22}
There exists $(u,v,w)$ such that for a subsequence (without change of notation) 
\begin{align} \label{weak2}
u_\e \longrightarrow u \quad &\mbox{weakly  in}\quad L^{2}(0,T:L^{2}(\Omega)) \,,\\ \label{ustr}
u_\e \longrightarrow u \quad &\mbox{a.e. and in}\quad L^{\frac{4}{3}}(0,T:L^{\frac{4}{3}}(\Omega)) \,,\\ \label{uee2}
u_\e^2 \longrightarrow u^2 \quad &\mbox{weakly in}\quad L^{1}(0,T:L^{1}(\Omega)) \,,\\
 \label{vstr}
v_\e \longrightarrow v \quad &\mbox{a.e. and in}\quad L^2(0,T:W^{1,2}(\Omega)) \,,\\
\label{vwekt}
v_{\e,t} \longrightarrow v_t \quad &\mbox{weakly in}\quad L^2(0,T:L^2(\Omega))\,, \\
\label{wstr}
\nabla w_\e \longrightarrow \nabla v \quad &\mbox{strongly in}\quad L^2(0,T:L^{2}(\Omega))\,.
\end{align}
\end{theorem}
\begin{proof}
The weak convergence (\ref{weak2}) follows from the first inequality in (\ref{C1C2}) and in view of (\ref{L43}) and (\ref{Lt43}), (\ref{ustr}) follows from the Aubin-Lions compactness lemma. The condition (\ref{ln43}) implies the equiintegrability of $\{u_\e^2\}$ in $\Omega_T$ which, in turn, produces its weak convergence in $L^1(\Omega_T)$, i.e. (\ref{uee2}). 
Next, again the Aubin-Lions lemma enables us to deduce from (\ref{v2})-(\ref{v3}) and (\ref{v5}) that  for a subsequence still denoted by $\e_j$
\begin{equation}\label{v2str}
v_{\e_j}\longrightarrow v\quad\mbox{in}\quad L^2(0,T:L^2(\Omega))\quad\mbox{as}\quad \e_j\searrow 0\,.
\end{equation}
Next, in order to obtain information on the convergence of gradients, let us denote the differences $V_{\varepsilon} = v_{\e_j}- v_{\e_k}$ and $U_{\varepsilon} = u_{\e_j}- u_{\e_k}$. With this notation, from the $v$-equation, we obtain: 
\begin{equation}\label{Vru1}
V_{\epsilon t} -   \Delta V_{\varepsilon}
=\mu_2  V_{\epsilon} (1- a_2 u_{\e_j} -v_{\e_j}) - \mu_2  v_{\e_k}( a_2U_{\e} + V_{\e}   ).
\end{equation}
On multiplying  this equation by $V_{\e}$ and then  integrating  by parts, we find, using  the non-negativity of solutions,  the H\"{o}lder inequality  and (\ref{vwinf}) 
\begin{align*}
    d_v\int_0^T\int_\Omega |\nabla V_\e|^2dxdt &\leq  \mu_2\int_0^T\int_\Omega |V_\e|^2dxdt +\mu_2a_2\int_0^T\int_\Omega v_{\e_k} |U_\e||V_\e|dxdt \\
    &\leq \mu_2\int_0^T\int_\Omega |V_\e|^2dxdt +\mu_2a_2\|v_{\e_k}\|_\infty\|U_\e\|_{L^{\frac{4}{3}}(0,T:L^{\frac{4}{3}}(\Omega))}\|V_\e\|_{L^4(0,T:L^4(\Omega))}\\
    &\leq \mu_2\int_0^T\int_\Omega |V_\e|^2dxdt +2\mu_2a_2\bar{v}_\infty^2\|_\infty\|U_\e\|_{L^{\frac{4}{3}}(0,T:L^{\frac{4}{3}}(\Omega))}.
\end{align*}
Notice that $\int_\Omega |V_\e(x,0)| dx=0$. It follows that $\{\nabla v_{\e_j}\}$ is a Cauchy sequence in $L^2(0,T:L^2(\Omega))$ provided 
\[v_{\e_j}\longrightarrow v\quad\mbox{in}\quad L^2(0,T:L^2(\Omega))\quad\mbox{and}\quad  u_{\e_j}\longrightarrow u\quad\mbox{in}\quad L^{\frac{4}{3}}(0,T:L^\frac{4}{3}(\Omega)). 
\]
Thus,
\begin{equation}
\nabla v_{\e_j}\longrightarrow \nabla v\quad\mbox{in}\quad L^2(0,T:L^2(\Omega)).
\end{equation}
Hence, (\ref{vstr}) follows in view of (\ref{v2str}).
The weak convergence in (\ref{vwekt}) is an immediate result of the uniform bound (\ref{v5}). 
Finally, to prove (\ref{wstr}) one notices first that after integration of (\ref{31-32}) from $0$ to $T$ in view of (\ref{v4}) and (\ref{v5}) we obtain
\[\int_0^T\int_\Omega |\nabla w_{\e_j}-\nabla v_{\e_j}|^2\longrightarrow 0 \quad\mbox{as}\quad {\e_j}\longrightarrow 0\]
and it remains to use the triangle inequality
\begin{align*}
\left(\int_0^T\int_\Omega |\nabla w_{\e_j}-\nabla v|^2dxdt\right)^{\frac{1}{2}}\leq& \left(\int_0^T\int_\Omega |\nabla w_{\e_j}-\nabla v_{\e_j}|^2 dxdt\right)^{\frac{1}{2}}\\   
&+ \left(\int_0^T\int_\Omega |\nabla v_{\e_j}-\nabla v|^2 dxdt\right)^{\frac{1}{2}}
\end{align*}
and (\ref{vstr}).
\end{proof}

Theorem \ref{Lw22} allows passing to the limit, $\lim_{j\rightarrow\infty}\e_j=0$, in all terms in the following weak formulation 
 \begin{align*}
  - \int_{0}^T \int_{\Omega} u_{\e_j} \varphi_t dxdt &-\int_\Omega u_0\varphi(x,0) dx
=   d_u \int_{0}^T \int_{\Omega}u_{\e_j} \Delta \varphi dxdt  
-\chi \int_{0}^T \int_{\Omega} u_{\e_j} \nabla w_{\e_j} \nabla \varphi dxdt  \\
&+\mu_1\int_{0}^T \int_{\Omega} u_{\e_j} ( 1- u_{\e_j} - a_1 v_{\e_j}) \varphi dxdt\\
  \int_{0}^T \int_{\Omega} v_{\e_j,t} \psi dxdt
&=  d_v \int_{0}^T \int_{\Omega} \nabla v_{\e_j} \nabla \psi  dxdt 
+  \mu_2\int_{0}^T \int_{\Omega}v_{\e_j}(1- v_{\e_j} - a_2 u_{\e_j}) \psi dxdt
    \end{align*}
for $\varphi \in C^\infty_0(\bar{\Omega}\times [0\,,T))$ and $\psi\in L^2(0,T: W^{1,2}(\Omega))$ where $C^\infty_0(\bar{\Omega}\times [0\,,T))$ is the space of smooth functions with compact support in $\bar{\Omega}\times [0\,,T) $.
 It is worth noticing that thanks to the $L^\infty(\Omega)$-estimates of $u_\e$ from Lemma \ref{Cor2} we have a stronger result for case $N\leq 2$ and the already established estimates allow us to use the standard compactness technique in a straightforward way.
\begin{theorem} \label{lnu2} If  $N \leq2$ and (\ref{icreg})  with $T<\infty$ are satisfied then 
\begin{align}\label{lnn2} 
&\int_0^T \int_{\Omega}  u_\e^2 +\int_0^T \int_{\Omega} |\nabla u_\e|^2 \leq C \\
\label{lnt2} 
&\{u_{\e,t} \} \quad\mbox{bdd. in } \quad L^2(0,T:W^{1,2}(\Omega)^\star) \,.
\end{align}
Moreover, upon extracting a subsequence (without the change of the notation)
\begin{align}\label{ustr2}
u_\e \longrightarrow u \quad &\mbox{strongly in}\quad L^2(0,T:L^{2}(\Omega))\,, \\   \label{uw22}
u_{\e}^2 \longrightarrow u^2 \quad &\mbox{weakly in}\quad L^2(0,T:L^2(\Omega))\,,\\
\label{uwekt}
u_{\e,t} \longrightarrow u_t \quad &\mbox{weakly in}\quad L^2(0,T:W^{1,2}(\Omega)^\star)\,. 
\end{align}
\end{theorem}
\begin{proof}
Lemma \ref{Cor2} and (\ref{wD})  allow us to test the u-equation with $u_\e$ to obtain the standard $L^2$ energy estimate. Finally,  (\ref{lnt2}) follows in a standard way using (\ref{lnn2}) and (\ref{ukk}).
By (\ref{lnn2}) and (\ref{lnt2} and the Aubin-Lions compactness lemma, we infer that for a subsequence there holds (\ref{ustr2}). The weak convergence in (\ref{uw22}) results from 
the $L^\infty$-bound in Lemma \ref{Cor2} and  (\ref{ustr2}), while (\ref{uwekt}) is an immediate consequence of the uniform bounds (\ref{v5}) and (\ref{lnt2}). 
\end{proof}
In consequence, for $N\leq 2$ there exists a subsequence of $\{\e_j\}_{j\in \N}$   such that it is possible passing to the limit in the following $L^2$-weak formulation of the system 
\begin{align*}
  \int_{0}^T \int_{\Omega} u_{\e,t} \varphi dxdt
&=   d_u \int_{0}^T \int_{\Omega}\nabla u_{\e} \nabla \varphi dxdt  
-\chi \int_{0}^T \int_{\Omega} u_{\e} \nabla w_{\e} \nabla \varphi dxdt  \\
&+\mu_1\int_{0}^T \int_{\Omega} u_{\e} ( 1- u_{\e} - a_1 v_{\e}) \varphi dxdt\\
  \int_{0}^T \int_{\Omega} v_{\e,t} \psi dxdt
&=  d_v \int_{0}^T \int_{\Omega} \nabla v_{\e} \nabla \psi  dxdt 
+  \mu_2\int_{0}^T \int_{\Omega}v_{\e}(1- v_{\e} - a_2 u_{\e}) \psi dxdt
    \end{align*}
for $\varphi\,,\psi\in L^2(0,T: W^{1,2}(\Omega))$
to get in the limit that $(u\,,v)$ satisfies (\ref{uw1})-(\ref{vw2}).
 We underline that, so far, the pair $(u,v)$ is merely a weak $L^2$ solution to the system. 
It is worth mentioning that the compatibility condition (\ref{ic}) is not needed to use the compactness method, and the above passage to the limit for $N\leq 2$ is free of this restriction. On the other hand, suitable assumptions on the regularity of the initial data and the classical regularity theory of parabolic systems lead to the alternative proof of classical solutions to the competition problem with repulsive competitor taxis (\ref{uu1})-(\ref{uu2}).

\begin{corollary} \label{lnu2cor} As a consequence of Theorem
\ref{lnu2} and uniqueness of the solution of the limit problem, cf. Lemma \ref{uniq}, we infer that, in fact, the full sequence converges to the same limit.
\end{corollary}

\begin{theorem}\label{reg} 
If $u_0\,, v_0\in C^{2+ \alpha}(\bar{\Omega})$ and $N\leq 2$, then the solution to (\ref{uw1})-(\ref{vw2}) satisfies
\[(u,v)\in (C^{2+\alpha,1+\frac{\alpha}{2}}(\bar{\Omega}\times[0\,,T))^2\].
\end{theorem}
\begin{proof}
First, due to the fact that $v(1- v - a_2 u)\in L^p(0,T:L^p(\Omega))$  for any $p\in [1\,,\infty]$ we infer by the $L^p$- maximal regularity (see for instance Theorem 48.2 and Remark 48.3 (ii) \cite{qs} p. 439) that, in fact, $v$ is a strong solution to the $v$-equation
\[v\in W^{1,p}(0,T:L^p(\Omega))\cap L^p(0\,,T:W^{2,p}(\Omega))\,\]
and also   $\nabla v\in L^\infty(0,T:L^\infty(\Omega))$, (see (\ref{zip})) 
and therefore $$ - \chi \nabla u \nabla v - \chi u  \Delta v\in L^{q}( 0, T: L^q(\Omega)) \mbox{ for  } q \leq 2  $$
 therefore, 
$$ u_t -\Delta u \in L^2( 0, T: L^2( \Omega)),$$
which implies $$u \in  W^{1,2}(0,T:L^2(\Omega))\cap L^2(0\,,T:W^{2,2}(\Omega))$$ and it allows us to rise the regularity of the weak solution in the classical way \cite[Chapter VII]{Lady} (see also   Remark 48.3 (ii) \cite{qs} p. 439).
\end{proof}
\section{Prey-predator  system with prey-taxis} \label{sec5}
In this section we use some tools developed in the previous sections to show that  the solution to the following prey-predator system with indirect prey taxis  studied in \cite{Ahnand} and \cite{Ahnand2} being the extension of model from \cite{TWr1}
\begin{align} \label{ue1z}
   z_{\e,t} - d_z\Delta z_\e &= -\chi \nabla \cdot (z_\e \nabla w_\e ) +\mu_1 z_\e -\mu_1'z_\e^2 + bF(v_\e)z_\e\,,\\ \label{ue2v}
   v_{\e,t} - d_v\Delta v_\e &= \mu_2 v_\e(1 -v_\e) -F(v_\e)z_\e\,,\\ \label{ue3w}
  \e w_{\e,t} - \e \Delta w_\e &=  v_\e- w_\e\,,
 \end{align}
converges to the solution of prey-predator model with prey-taxis studied in many papers among which we indicate \cite{JinWang} with references given there
 \begin{align} \label{uu1z}
   z_{t} - d_z\Delta z &= -\chi \nabla \cdot (z \nabla v) +\mu_1z -\mu_1'z^2+ bF(v)z\,,\\ \label{uu2v}
   v_{t} - d_v\Delta v &= \mu_2 v(1 -v)-F(v)z\,
 \end{align}
with suitably regular initial conditions and homogeneous Neumann boundary conditions  for both systems. 

The growth rate coefficient $\mu_1$ in most predator-prey models is reduced to the death rate $\mu_1<0$ since the birth rate is usually due only to prey consumption, however, we only assume that $\mu_1\in\R$ along with the positivity of the intraspecific competition coefficient $\mu_1'>0$. The last assumption is crucial for our analysis, though in basic predator-prey models, it is assumed that $\mu_1<0$ and $\mu_1'=0$.
The function $F:\R_+\rightarrow \R_+$, assumed to be smooth, is usually referred to as the functional response and it satisfies the  assumption: 
\begin{equation}\label{Funct}
F(\xi)\leq C_F\xi, \quad\mbox{for}\quad \xi\geq 0, \; \;  \mbox{ and }  |F(\xi_1) - F( \xi_2) | \leq L |\xi_1 - \xi_2| , \; \; \; 
\mbox{ for }\xi_1, \xi_2  \geq 0
\end{equation}
with  a positive constant $C_F$ and Lipschitz constant $L$. We note that condition (\ref{Funct}) is satisfied by most biologically relevant functional responses, including three types of Holling function (cf. \cite{Hol}).  The constant $b$ is the efficiency with which the consumed prey is converted into predator offspring.

It turns out that the passage from (\ref{ue1z})-(\ref{ue3w}) to (\ref{uu1z})-(\ref{uu2v}) demands only small modifications of the arguments used in the previous sections for the case of competition model.  It is worth noticing that the existence of global classical solutions to the system (\ref{uu1z})-(\ref{uu2v}) (known so far only for $N\leq 2$) is much more difficult to prove in comparison to the system (\ref{ue1z})-(\ref{ue3w}) (see \cite{Ahnand}, \cite{Ahnand2} and \cite{JinWang}).
 
 The uniform $L^1(\Omega)$-bound in (\ref{M1}) for the first two unknown still holds, but requires a different argument.  
\begin{lemma} \label{L11}
 We assume that the initial conditions  $$z_{\e,0}\,,v_{\e,0}\,,w_{\e,0}\in W^{1,q} (\Omega)\,, q>\max\{N\,,4\}$$ are nonnegative functions.   Then,  for any $\e\in (0,1)$ there exists a global in time unique  classical solution  to (\ref{ue1z})-(\ref{ue3w})  defined in $\Omega\times (0\,,+\infty)$ such that for any $T>0$ 
\[(u_\e,v_\e,w_\e)\in (C([0\,,T]:W^{1,q}(\Omega))\cap 
C^{2,1}(\bar{\Omega}\times (0\,,T)))^3\,.\] 
Moreover, there exist constants $\bar{z}_1\,,\bar{v}_1\,,\bar{w}_1$, $C_1(T)$,   $C_2(T)$, $\bar{v}_\infty$ and $\bar{w}_\infty$ such that,  for any $\e\in (0,1)$
\begin{equation}\label{M11}
\sup_{t>0}\left(\Vert z_\e(\cdot,t)\Vert_{1}\right)\leq \bar{z}_1\,,\quad 
\sup_{t>0}\left(\Vert v_\e(t)\Vert_{1}\right)\leq \bar{v}_1\,,\quad \sup_{t>0}\left(\Vert w_\e(t)\Vert_{1}\right)\leq \bar{w}_1, 
\end{equation}
\begin{equation}\label{vwinf1}
\sup_{t>0}\left(\Vert v_\e(t)\Vert_{\infty}\right)\leq \bar{v}_\infty\,,\quad \sup_{t>0}\left(\Vert w_\e(t)\Vert_{\infty}\right)\leq \bar{w}_\infty\,, 
\end{equation}
\begin{equation}\label{C1C21}
   \int_{0}^T\int_\Omega z_\e(x,s)^2dxds\leq  C_1(T)\,,\quad \int_{0}^T\int_\Omega v_\e(x,s)^2dxds\leq  C_2(T)\,
\end{equation} 	
and 
\begin{equation}\label{t11}
    \int_{t}^{t+1}\int_\Omega z_\e(x,s)^2dxds\leq  (\mu_1+C_F\bar{v}_\infty)\bar{z}_1\,,\; \int_{t}^{t+1}\int_\Omega v_\e(x,s)^2dxds\leq \mu_2\bar{v}_1\,
\end{equation} for any  $t>0$.
\end{lemma}
\begin{proof}
To prove (\ref{M11}) we multiply the $v$-equation by $b$, integrate in $\Omega$ and sum up the first two equations to obtain 
\begin{align*}
 &\frac{d}{dt}\left(\int_\Omega z_\e + b\int_\Omega v_\e\right) +\left(\int_\Omega z_\e + b\int_\Omega v_\e\right) \\
 &\leq (\mu_1+1)_+\int_\Omega  z_\e -\mu_1'\int_\Omega z_e^2 + b(\mu_2 +1)\int_\Omega  v_\e -\mu_2\int_\Omega v_e^2   
\end{align*}
and using Jensen's inequality we obtain
\begin{align*}
& \frac{d}{dt}\left(\int_\Omega z + b\int_\Omega v\right) +\left(\int_\Omega z + b\int_\Omega v\right) \leq\\
& (\mu_1+1)_+\int_\Omega  z_\e  -\frac{\mu_1'}{|\Omega|}\left(\int_\Omega z_\e \right)^2 +b(\mu_2 +1)\int_\Omega  v_\e  -\frac{b\mu_2}{|\Omega|}\left(\int_\Omega v_\e \right)^2 \\
&\leq \frac{|\Omega|}{4}\left(\frac{(\mu_1+1)_+^2}{\mu_1'}+\frac{b^2(\mu_2+1)^2}{\mu_2}\right).
\end{align*}
Hence (\ref{M11}) easily follows from the differential inequality while the proofs of the remaining results are the same as that in Lemma \ref{L1}. \end{proof}
Notice that the terms $-a_2 \mu_2av_\e u_\e$ and $-F(v_\e) z_\e$ in (\ref{ue2}) and (\ref{ue2v}) respectively have the same sign and growth and therefore the lemmata \ref{base}, \ref{we}, \ref{L6}, \ref{W4} and \ref{Cor1} that only concern the last two unknowns still hold true for the system (\ref{ue1z})-(\ref{ue3w}). Moreover, without essential changes in the proof, we infer that the
$L^\infty(\Omega)$-bound for $z_\e$ holds similarly to the lemma \ref{Cor2}.
\begin{lemma}\label{Cor21}
 For space dimension $N\leq 2$ there is a constant $\bar{z}_\infty$ such that 
 \begin{equation} \label{u121}
     \sup_{t>0}\Vert z_\e\Vert_\infty \leq \bar{z}_\infty\,.
 \end{equation}
\end{lemma}
\begin{proof}
 Notice first that the proof of Lemma \ref{uk} demands only minor modifications related to the term $C_F\bar{v}_\infty\int_\Omega z^k_\e$ that ultimately lead to (\ref{llap}) for the case $N=2$ and to (\ref{llap1}) for the case $N=1$. Therefore, we may proceed in the same way as in the proof of Lemma \ref{Cor1}, the only difference being the use of inequality (\ref{zinf}), because now $f_+=\frac{(\mu_1 +C_F\bar{v}_\infty)^2}{4}$.
\end{proof}
To adapt Lemma \ref{UVW} to the case of prey-predator model we only replace $u_\e$ by $z_\e$ and 
$U_\e$ by $Z_e=z_e-z$ with minor changes in the proof, which amounts to replacing (\ref{Us1}) by 
 \begin{align*}
Z_{\varepsilon t}- d_z\Delta Z_{\varepsilon } =&  
\nabla \cdot \left(\chi ( Z_{\e} \nabla  w_\e  + z \nabla  W_{\e} \right) \\
&+ (\mu_1 -\mu_1'(z_\e+z))Z_{\e} +F(v_\e)Z_\e +z(F(v_e)-F(v)) Z_\e
\end{align*}
and noticing that after  multiplying  by $Z_{\varepsilon}$ and integrating  by parts we obtain qualitatively the same inequality as (\ref{Us1}) namely 
  \begin{align*}
 \frac{1}{2} \frac{d}{dt} 
 \int_{\Omega} Z_{\varepsilon}^2
 + \frac{d_z}{2} 
 \int_{\Omega} |\nabla Z_{\varepsilon}|^2
 \leq & \displaystyle   \frac{\chi^2}{d_z} \left(
 \int_{\Omega} Z_{\varepsilon}^2 |\nabla w_{\varepsilon} |^2
 +  \int_{\Omega} z^2 |\nabla W_{\varepsilon} |^2\right) \\ 
&+(\mu_1 + C_F\bar{v}_\infty) \int_{\Omega} Z_{\e}^2  +  \bar{z}_\infty L \int_{\Omega} Z_\e V_\e dx \,.
 \end{align*}
Next, only minor changes are needed to adapt the remaining part of the proof of Lemma \ref{UVW} and that of Theorem \ref{L22} in which it is assumed that 
\begin{equation}\label{ic1}
 z(\cdot,0)=z_\e(\cdot,0)=u_0(\cdot)\,, v(\cdot,0)=v_\e(\cdot,0)=w_\e(\cdot,0)=v_0(\cdot,0) .
 \end{equation} 
After adaptation of results from section 4 to the case of predator-prey model we obtain the following result 
\begin{theorem} \label{TheoPred}
\noindent
\begin{itemize}
    \item[i)] If  $N\leq 2$ and (\ref{ic1}) is satisfied then letting $\e\rightarrow 0$ in the weak formulation of (\ref{ue1z})-(\ref{ue3w}) one  obtains the weak solution $(z,v)$
    to (\ref{uu1z})-(\ref{uu2v}) which satisfies for any $T>0$
    $$z,v\in W^{1,2}(0,T:L^2(\Omega))\cap L^2(0\,,T:W^{1,2}(\Omega))\cap L^\infty(0,T:L^\infty(\Omega)$$ 
 and 
 \begin{align}\label{zw1}
     \int_{0}^T \int_{\Omega} z_{t} \varphi &= -d_z \int_{0}^T \int_{\Omega} \nabla z \nabla \varphi
+ \chi \int_{0}^T \int_{\Omega} z \nabla v \nabla \varphi \\ \nonumber
& +  \int_{0}^T \int_{\Omega} (\mu_1 z - \mu_1'z^2 + bF(v)z) \varphi\\ \label{vw21}
 \int_{0}^T \int_{\Omega} v_{t} \psi 
=&   \int_{0}^T \int_{\Omega} \nabla v \nabla \psi +
\int_{0}^T \int_{\Omega} (\mu_2v(1-v) - F(v)z) ) \psi
\end{align}
 for all $\varphi\,,\psi\in L^2(0\,, T: W^{1,2}(\Omega))$.
 \item[ii)] If only $N\leq 2$ then the weak limit $(z,v)$ in the aforementioned sense is attained in the limit for a subsequence $\{\e_j\}_{j\in\N}$ extracted from $\e$.
 \item[iii)]If $N\geq 3$ then there is a subsequence extracted from $\e$ which converges to $(z,v)$ satisfying (\ref{uu1z})-(\ref{uu2v}) in the following very weak sense 
 \begin{align*}
  - \int_{0}^T \int_{\Omega} z \varphi_t dxdt &-\int_\Omega z_0\varphi(x,0) dx
=   d_z \int_{0}^T \int_{\Omega}z \Delta \varphi dxdt  
+\chi \int_{0}^T \int_{\Omega} z \nabla v \nabla \varphi dxdt  \\
&+\int_{0}^T \int_{\Omega} (\mu_1 z - \mu_1'z^2 + bF(v)z) \varphi dxdt\\
  \int_{0}^T \int_{\Omega} v_{t} \psi dxdt
&=  d_v \int_{0}^T \int_{\Omega} \nabla v\nabla \psi  dxdt 
+  \mu_2\int_{0}^T \int_{\Omega}(\mu_2v(1-v) - F(v)z)) \psi dxdt
    \end{align*}
for $\varphi \in C^\infty_0(\bar{\Omega}\times [0\,,T))$ and $\psi\in L^2(0,T: W^{1,2}(\Omega))$ where $C^\infty_0(\bar{\Omega}\times [0\,,T))$ is the space of smooth functions with compact support in $\bar{\Omega}\times [0\,,T)$. 
\end{itemize}  
\end{theorem}
The same arguments as in Theorem \ref{reg}, based on the regularity theory of parabolic systems, lead to the following result which provides an alternative proof of the existence of solutions to the predator-prey system for $N\leq 2$.    
\begin{corollary}
 $z_0\,, v_0\in C^{2+\alpha}(\bar{\Omega})$ and $N\leq 2$, then the solution to (\ref{zw1})-(\ref{vw21}) satisfies
\[(z,v)\in (C^{2+\alpha,1+\frac{\alpha}{2}}(\bar{\Omega}\times [0\,,T)))^2. \]
\end{corollary}

\end{document}